\documentclass[11pt]{article}

\usepackage{amsmath}
\usepackage{amsthm}
\usepackage{amssymb}
\usepackage{amsfonts}
\usepackage{url}
\usepackage{verbatim}
\usepackage{graphicx}
\usepackage{subfigure}
\usepackage{dsfont}
\usepackage{color}
\usepackage{fancybox}
\usepackage{booktabs}
\usepackage{authblk}
\usepackage{tikz}
\usepackage{tikz-network}
\usepackage{wrapfig}
\usepackage{mathtools}

\theoremstyle{plain}  
\newtheorem{thm}{Theorem}[section]

\theoremstyle{definition}  
\newtheorem{defn}[thm]{Definition}

\theoremstyle{remark}  

\newtheorem{rem}[thm]{Remark}

\providecommand{\keywords}[1]{\textbf{\textit{Keywords:}} #1}

\newcommand{\rr}{\mathbb{R}}
\newcommand{\R}{{\mathbb{R}}}

\newcommand{\setof}[1]{\left\{ {#1}\right\}}

\newcommand{\sM}{{\mathsf M}}

\newcommand{\cF}{{\mathcal F}}
\newcommand{\cP}{{\mathcal P}}
\newcommand{\cQ}{{\mathcal Q}}
\newcommand{\cR}{{\mathcal R}}
\newcommand{\cT}{{\mathcal T}}

\def\mapright#1{\stackrel{#1}{\longrightarrow}}

\let\setof\relax
\DeclarePairedDelimiter{\setof}{\{}{\}}  
\DeclarePairedDelimiter{\paren}{(}{)}
\DeclarePairedDelimiter{\brak}{[}{]}
\renewcommand{\forall}{\text{ for all }}

\title{Extending combinatorial regulatory network modeling to include activity control and decay modulation}

\author[,a]{Bree Cummins\thanks{breschine.cummins@gmail.com}}
\author[,b,c]{Marcio Gameiro\thanks{gameiro@math.rutgers.edu}}
\author[,a]{Tom\'{a}\v{s} Gedeon\thanks{gedeon@math.montana.edu}}
\author[,d]{Shane Kepley\thanks{s.kepley@vu.nl}}
\author[,b]{Konstantin Mischaikow\thanks{mischaik@math.rutgers.edu}}
\author[,b]{Lun Zhang\thanks{lz210@rutgers.edu}}

\affil[a]{Department of Mathematical Sciences, Montana State University, Bozeman, MT, 59715}
\affil[b]{Department of Mathematics, Rutgers, The State University of New Jersey, Piscataway, NJ, 08854}
\affil[c]{Instituto de Ci\^{e}ncias Matem\'{a}ticas e de Computa\c{c}\~{a}o, Universidade de S\~{a}o Paulo, Caixa Postal 668, 13560-970, S\~{a}o Carlos, SP, Brazil}
\affil[d]{Department of Mathematics, VU Amsterdam, 1081 HV Amsterdam, The Netherlands}

\begin{document}

\maketitle

\begin{abstract}
Understanding how the structure of within-system interactions affects the dynamics of the system is important in many areas of science. We extend a network dynamics modeling platform DSGRN, which combinatorializes both dynamics and parameter space to construct finite but accurate summaries of network dynamics, to new types of interactions. While the standard  DSGRN assumes that each network edge controls the rate of abundance of the target node, the new edges may control either activity level or a decay rate of its target. While motivated by processes of post-transcriptional modification and ubiquitination in systems biology, our extension is applicable to the dynamics of any signed directed network.
\end{abstract}



\section{Introduction}
\label{sec:intro}

Networks have become a paradigm for organizing relational information: each node is associated with a particular object or quantity and edges indicate a relation between these objects or quantities.
In many cases these relations are meant to capture causality, e.g.\ a directed  edge from node $m$ to node $n$ indicates that the product associated with node $m$ has an impact on the product associated with node $n$.
We refer to such a network as a \emph{regulatory network}.
Our main goal is to identify possible dynamics of a given a regulatory network.
In order to better specify the problem and make it relevant for applications the following challenges need to be addressed.
\begin{description}
\item[C1] The network structure must encode a sufficiently broad range of meaningful causal interactions so that the range of dynamics relevant to applications can be realized.
\item[C2] The computational framework should take the network structure as input and output an identification and characterization of global dynamics.
\item[C3] There needs to be a theoretical framework that ties the outputs of the computations back to the dynamics of the application of interest.
\end{description}
 
While the focus of this paper is on {\bf C2}, the motivation for this work is the analysis of regulatory networks arising from systems biology.
Thus, we will partially address {\bf C1} in the context of gene regulatory networks that include post-transcriptional modifications such as phosphorylation and ubiquitination. 
The theoretical validation of our approach, i.e.\ {\bf C3}, is based on previous and ongoing work \cite{kalies:mischaikow:vandervorst:14,cummins:gedeon:harker:mischaikow:mok,gedeon:cummins:harker:mischaikow:plos,gameiro:gedeon:kepley:mischaikow,gedeon:2020} and is discussed at relevant points in the paper.

We make two assumptions that are maintained throughout the paper.
\begin{description}
\item[A1] An ordinary differential equation (ODE) provides an adequate model for the dynamics.
\item[A2] Let $x_n$ denote the quantity of product associated with node $n$.
Then, the rate of change of $x_n$ can be expressed as 
\begin{equation}
\label{eq:generalNonlinearity}
    -\Gamma_n(x)x_n + \Lambda_n(x)
\end{equation}
where $\Gamma_n(x)$ and $\Lambda_n(x)$ quantify the rate of decay  and the rate of production of $x_n$, 
respectively. 
The network encodes the coordinates of  the state  $x$ upon which $\Gamma_n$ and $\Lambda_n$ are dependent, but we do not assume a particular functional form for this dependency.
\end{description}
Observe that we have refrained from writing \eqref{eq:generalNonlinearity} in the form of an ODE. 
This is to emphasize the fact that we use explicit functions for $\Gamma$ and $\Lambda$ only for computational purposes, but we are not interested in the dynamics  in the traditional sense, e.g.~trajectories or equilibria of the resulting differential equations. Instead, the appropriate interpretation of the dynamics obtained via our computations should be derived indirectly from associated lattice structures and algebraic topology, which is  part of {\bf C3}.  

Our approach to {\bf C2} is an extension to an earlier approach based on a combinatorial representation of the dynamics \cite{cummins:gedeon:harker:mischaikow:mok}.
At its foundation  lies the perspective that because we do not know the precise nonlinearities we should not try to identify dynamics on the level of trajectories.
Instead the goal is to provide a computationally efficient  robust combinatorial representation of the dynamics that, at a minimum, is capable of accurately identifying  existence and structure of attractors. 
Furthermore, since the dynamics that can be exhibited by a network is parameter dependent, it is desirable that there is a clear correspondence between parameters and global dynamics.
With this in mind we developed the Dynamic Signatures Generated by Regulatory Networks (DSGRN) software \cite{cummins:gedeon:harker:mischaikow:mok,Cummins2017b,gedeon:2020} that takes a network as input, creates an appropriate parameter space along with an explicit finite decomposition thereof, and  for each region of parameter space computes a combinatorial/algebraic topological description of the global dynamics (see Section~\ref{sec:dsgrn} for further details).

This approach has been applied to a variety of regulatory networks associated with questions and challenges from systems and synthetic biology including: identification of oscillatory behavior in a simple model of the p53 network  \cite{cummins:gedeon:harker:mischaikow:mok}, identification of minimal models for the switching behavior of the mammalian Rb-E2F system \cite{gedeon:cummins:harker:mischaikow:plos}, EMT \cite{xin:cummins:gedeon}, oocyte \cite{diegmiller}, and design of optimal 3 node hysteretic switches \cite{gameiro:gedeon:kepley:mischaikow}.
This variety of applications suggests that DSGRN is a potentially powerful tool for the global analysis of networks.
However, in the above mentioned biological contexts the current version of DSGRN imposes two significant constraints.
The first is that the decay rate $\Gamma_n$ is assumed to be constant, i.e.\ not controlled by other  nodes (representing protein concentrations)  within the network.
However, the fact that the 2004 Nobel Prize in Chemistry was awarded for the discovery that ubiquitination leads to protein decay \cite{hershko:ciechanover:rose, nobel} indicates that this is a rather severe assumption.
The second limitation is that, once produced, a protein has a constant efficacy with which it controls the activation or repression of  its targets.
This ignores the common phenomenon of protein-protein interactions that can have dramatic effect on the regulatory capabilities of the targeted protein and decisive effect on network function~\cite{ricci-tam}.
In other words, the current DSGRN does not allow for general enough interpretations/expression of  both $\Gamma_n$ and  $\Lambda_n$, i.e.\ it is lacking with respect to {\bf C1}.

Observe that the above mentioned biochemical constraints are associated with how the output of one node affects the  activity of another node and thus are expressed via edges within the regulatory network.
With this in mind in this paper we extend the DSGRN 
framework and the corresponding  software to allow for regulatory networks with edges of the form described in Figure~\ref{fig:RN+}.
More detailed descriptions of the meaning of this notation is provided in
Sections~\ref{sec:dsgrn} and \ref{sec:extension}.
For the moment it is sufficient to know that the diagrams indicate edges with the following properties.
\begin{enumerate}
\item The \underline{solid (type $0$)} edges of Figure~\ref{fig:RN+}(a) indicate that an increase in $x_1$ leads to a higher rate of production of $x_2$, while an increase in $x_3$ leads to a decrease in production of $x_4$.  
Regulatory networks consisting only of these types of edges can be handled by the original DSGRN software.
\item The \underline{dashed (type I)} edges indicated in Figure~\ref{fig:RN+}(b) indicate that an increase in $x_1$ increases the decay rate of $x_2$ while an increase in $x_3$ decreases the decay rate of $x_4$.
\item The \underline{type II} edges are pairs of edges where one edge connects a node to the other edge. A solid dot on the edge in Figure~\ref{fig:RN+}(c) indicates that the product from node 1 needs to be modified to become active. 
The pointed arrow from node 3 to the edge indicates that an increase in $x_3$  leads to an increase in the fraction of $x_1$ that is modified, while the blunt arrow from node 2 indicates that an increase in $x_2$ leads to a decrease in the fraction of $x_1$ that is modified.
The pointed arrow from node 1 indicates that once the modification occurs the product of 1 acts as an activator. 
The blunt arrow from node $1'$ indicates that once modified the output from node $1'$ acts as a repressor.
\item The empty dot on the receiving edge in Figure~\ref{fig:RN+}(d), also type II edges, indicates that the product from node 1 is active, but that it can be modified to be deactivated.
In particular, the blunt edge of node 3 indicates that $x_3$ decreases the fraction of $x_1$ that is  modified, while the pointed arrow of node 2 indicates that $x_2$ maintains or increases the fraction of $x_1$ that is modified.
\item We conclude the description by noting that activity  modifications described in Figure~\ref{fig:RN+}(c)-(d) can be also applied to dashed arrows from  Figure~\ref{fig:RN+}(b). That is, the solid and empty dot can be also  placed on dashed edges.
\end{enumerate}
The biologically minded reader may wish to identify Figure~\ref{fig:RN+}(b)  with the process of ubiquitination and processes in Figure~\ref{fig:RN+}(c) and (d) 
with phosphorylation and dephosphorylation, or other post-transcriptional and post-translational protein modification.

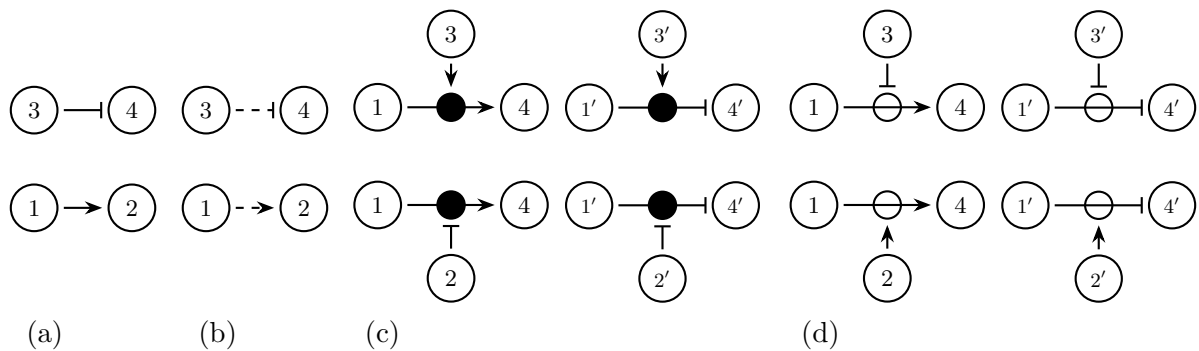
\begin{figure}[!htbp]
\begin{picture}(400,140)
\put(10,0) {(a)}
\put(0,15){
\begin{tikzpicture}[thick, main node/.style={circle, fill=white, draw, font=\footnotesize}, fake node/.style={circle}, scale=0.65]

\node[main node] (1) at (0,1.5) {$1$};
\node[main node] (2) at (2,1.5) {$2$};
\node[main node] (3) at (0,3.5) {$3$};
\node[main node] (4) at (2,3.5) {$4$};
\node[fake node] (5) at (1,0) {~~~};

\draw[-{Stealth}, shorten <= 2pt, shorten >= 2pt] (1) -- (2);
\draw[-|, shorten <= 2pt, shorten >= 2pt] (3) -- (4);
\end{tikzpicture}
}
\put(75,0){(b)}
\put(65,15){
\begin{tikzpicture}[thick, main node/.style={circle, fill=white, draw, font=\footnotesize}, fake node/.style={circle}, scale=0.65]

\node[main node] (1) at (0,1.5) {$1$};
\node[main node] (2) at (2,1.5) {$2$};
\node[main node] (3) at (0,3.5) {$3$};
\node[main node] (4) at (2,3.5) {$4$};
\node[fake node] (5) at (1,0) {~~~};

\draw[-{Stealth}, dashed, shorten <= 2pt, shorten >= 2pt] (1) -- (2);
\draw[-|, dashed, shorten <= 2pt, shorten >= 2pt] (3) -- (4);
\end{tikzpicture}
}
\put(138,0){(c)}
\put(130,15){
\begin{tikzpicture}[thick, main node/.style={circle, fill=white, draw, font=\footnotesize}, mid+ node/.style={circle, fill=black, draw}, scale=0.65]

\node[main node] (1) at (0,1.5) {$1$};
\node[main node] (2) at (1.5,0) {$2$};
\node[main node] (3) at (1.5,5) {$3$};
\node[main node] (4) at (3,1.5) {$4$};
\node[mid+ node] (5) at (1.5,1.5) {};
\node[main node] (6) at (0,3.5) {$1$};
\node[main node] (7) at (3,3.5) {$4$};
\node[mid+ node] (8) at (1.5,3.5) {};

\draw[-{Stealth}, shorten <= 2pt, shorten >= 2pt] (1) -- (4);
\draw[-|, shorten <= 2pt, shorten >= 2pt] (2) -- (5);
\draw[-{Stealth}, shorten <= 2pt, shorten >= 2pt] (6) -- (7);
\draw[-{Stealth}, shorten <= 2pt, shorten >= 2pt] (3) -- (8);
\end{tikzpicture}
}
\put(210,15){
\begin{tikzpicture}[thick, main node/.style={circle, fill=white, draw, font=\footnotesize}, mid+ node/.style={circle, fill=black, draw}, scale=0.65]

\node[main node, scale=0.85] (1) at (0,1.5) {$1'$};
\node[main node, scale=0.85] (2) at (1.5,0) {$2'$};
\node[main node, scale=0.85] (3) at (1.5,5) {$3'$};
\node[main node, scale=0.85] (4) at (3,1.5) {$4'$};
\node[mid+ node] (5) at (1.5,1.5) {};
\node[main node, scale=0.85] (6) at (0,3.5) {$1'$};
\node[main node, scale=0.85] (7) at (3,3.5) {$4'$};
\node[mid+ node] (8) at (1.5,3.5) {};

\draw[-|, shorten <= 2pt, shorten >= 2pt] (1) -- (4);
\draw[-|, shorten <= 2pt, shorten >= 2pt] (2) -- (5);
\draw[-|, shorten <= 2pt, shorten >= 2pt] (6) -- (7);
\draw[-{Stealth}, shorten <= 2pt, shorten >= 2pt] (3) -- (8);
\end{tikzpicture}
}
\put(303,0){(d)}
\put(295,15){
\begin{tikzpicture}[thick, main node/.style={circle, fill=white, draw, font=\footnotesize}, mid- node/.style={circle, draw}, scale=0.65]

\node[main node] (1) at (0,1.5) {$1$};
\node[main node] (2) at (1.5,0) {$2$};
\node[main node] (3) at (1.5,5) {$3$};
\node[main node] (4) at (3,1.5) {$4$};
\node[mid- node] (5) at (1.5,1.5) {};
\node[main node] (6) at (0,3.5) {$1$};
\node[main node] (7) at (3,3.5) {$4$};
\node[mid- node] (8) at (1.5,3.5) {};

\draw[-{Stealth}, shorten <= 2pt, shorten >= 2pt] (1) -- (4);
\draw[-{Stealth}, shorten <= 2pt, shorten >= 2pt] (2) -- (5);
\draw[-{Stealth}, shorten <= 2pt, shorten >= 2pt] (6) -- (7);
\draw[-|, shorten <= 2pt, shorten >= 2pt] (3) -- (8);
\end{tikzpicture}
}
\put(375,15){
\begin{tikzpicture}[thick, main node/.style={circle, fill=white, draw, font=\footnotesize}, mid- node/.style={circle, draw}, scale=0.65]

\node[main node, scale=0.85] (1) at (0,1.5) {$1'$};
\node[main node, scale=0.85] (2) at (1.5,0) {$2'$};
\node[main node, scale=0.85] (3) at (1.5,5) {$3'$};
\node[main node, scale=0.85] (4) at (3,1.5) {$4'$};
\node[mid- node] (5) at (1.5,1.5) {};
\node[main node, scale=0.85] (6) at (0,3.5) {$1'$};
\node[main node, scale=0.85] (7) at (3,3.5) {$4'$};
\node[mid- node] (8) at (1.5,3.5) {};

\draw[-|, shorten <= 2pt, shorten >= 2pt] (1) -- (4);
\draw[-{Stealth}, shorten <= 2pt, shorten >= 2pt] (2) -- (5);
\draw[-|, shorten <= 2pt, shorten >= 2pt] (6) -- (7);
\draw[-|, shorten <= 2pt, shorten >= 2pt] (3) -- (8);
\end{tikzpicture}
}
\end{picture}
\caption{(a) {\bf Type 0 edges.} Direct up regulation of 2 by 1 and direct down regulation of 4 by 3. (b) {\bf Type I edges.}  Node 1 increases decay rate of 2 and 3 decreases decay rate of 4. (c)  {\bf Type II edges.} Output of 1 has two states, unmodified and modified. The solid dot indicates that the product from node 1 needs to be modified to become active. The pointed arrow from 3 leads to modification of 1 (thus activates) and 2 leads to demodification of 1 (thus represses). (d) {\bf Type II edges} Output of 1 is has two states, unmodified and modified. The hollow dot indicates that the product from node 1 is active, but it can be modified to be deactivated. The blunt arrow from 3 leads to modification of 1 (thus represses) and 2 leads to demodification of 1 (thus activates).
}
\label{fig:RN+}
\end{figure}

Keeping in mind that the focus of this paper is on {\bf C2} the outline for this paper is as follows.
In Section~\ref{sec:dsgrn} we review the mathematics and combinatorics that provide the foundation for the current version of the DSGRN software.
As explained in this section, the most significant constraint on the software is the number of in and out-edges, as well as their \textit{interaction type} at any given node. The interaction type describes how the inputs at each node are combined to inform the node's output.
In Section~\ref{sec:extension} we describe the extension to DSGRN and via Tables~\ref{table:ubiquitination_logic} and \ref{table:phosphorylation_logic} we indicate the
interaction types
that the extended software can currently handle.
In principle these two sections provide sufficient information for a user to have a conceptual understanding of how the DSGRN software functions.

Of course, the DSGRN software has been constructed with the aim of being a useful tool in the analysis of regulatory networks that arise from applications. 
Thus in Section~\ref{sec:extension} we consider {\bf C1} and discuss ODE models arising from post-transcriptional regulation of gene regulatory networks and demonstrate how the computations that DSGRN performs can be identified with a singular limit of these systems.
In Section~\ref{sec:examples} we compare the range of global dynamics for networks based on edges of the type in Figure~\ref{fig:RN+}(a) against similar networks that allow for the full range of edge interaction, i.e.\ including Figures~\ref{fig:RN+}(b)-(d).

\section{DSGRN}
\label{sec:dsgrn}

The computational utility of the DSGRN software arises from two distinct combinatorial abstractions.
The first is a combinatorial representation of the dynamics.  
Furthermore, the expansion of DSGRN presented in this paper requires no adjustments to the previous versions regarding the representation of the dynamics and therefore we refer the reader to \cite{cummins:gedeon:harker:mischaikow:mok,gameiro:gedeon:kepley:mischaikow}
for details on the DSGRN approach.
The second is an explicit decomposition of parameter space into a finite collection of semi-algebraic sets, with the property that the combinatorial dynamics is constant for all parameters within each such semi-algebraic set.
A key point of this paper is the derivation of the proper decompositions for networks that contain edges of the form in Figure~\ref{fig:RN+}(b)-(d), and thus  we review the decomposition in this section. 

To explain how the DSGRN software decomposes parameter space we begin by considering the particularly simple network in Figure~\ref{fig:toggle}(a) (for complete details the reader is referred to  \cite{cummins:gedeon:harker:mischaikow:mok, gameiro:gedeon:kepley:mischaikow, kepley:mischaikow:zhang}).
Notice that the edges are type $0$ edges  shown in Figure~\ref{fig:RN+}(a).
As indicated in the introduction and compatible with the discussion of Figure~\ref{fig:RN+}, the decay rate for $x_n$ in \eqref{eq:generalNonlinearity} is assumed to be a positive constant that we denote by $\gamma_n$.
The rate of production of $x_1$ is given by $\Lambda_1(x_2)$ since the unique in-edge to node 1 comes from node 2. 
Similarly, the rate of production of $x_2$ is given by $\Lambda_2(x_1)$.
The edge $2 \dashv 1$ indicates that $x_2$ represses the production of $x_1$, while the edge $1 \rightarrow 2$ indicates that $x_1$ activates the production of $x_2$.
Step functions provide the simplest characterization of these phenomena, thus we introduce the functional expressions
\begin{equation}
    \label{eq:lambdaplusminus}
\lambda^+_{n,m}(x_m) := \begin{cases}
\ell_{n,m} & \text{if $x_m < \theta_{n,m}$} \\
\ell_{n,m} + \delta_{n,m}& \text{if $x_m > \theta_{n,m}$}
\end{cases}
\quad\text{and}\quad
\lambda^-_{n,m}(x_m) := \begin{cases}
\ell_{n,m} + \delta_{n,m} & \text{if $x_m < \theta_{n,m}$} \\
\ell_{n,m}& \text{if $x_m > \theta_{n,m}$}
\end{cases}
\end{equation}
where the parameters $\theta_{n,m}$, $\ell_{n,m}$, and $\delta_{n,m}$ are assumed to be positive.  
In particular, for the regulatory network of Figure~\ref{fig:toggle}(a) the rate of production of $x_1$ and $x_2$, given by \eqref{eq:generalNonlinearity}, takes the form
\begin{equation}
\label{eq:toggleSwitch}
    \begin{aligned}
  -\gamma_1 x_1 & + \lambda^-_{1,2}(x_2) \\
  -\gamma_2 x_2 & + \lambda^+_{2,1}(x_1) .
  \end{aligned}
\end{equation}
We associate four positive parameters to each node: $(\gamma_1,\theta_{2,1},\ell_{1,2},\delta_{1,2})$ to node 1, and $(\gamma_2,\theta_{1,2},\ell_{2,1},\delta_{2,1})$ to node 2.

Thus, the parameter space is $(0,\infty)^8 = (0,\infty)^4 \times (0,\infty)^4 $. The phase space is $(0,\infty)^2$.

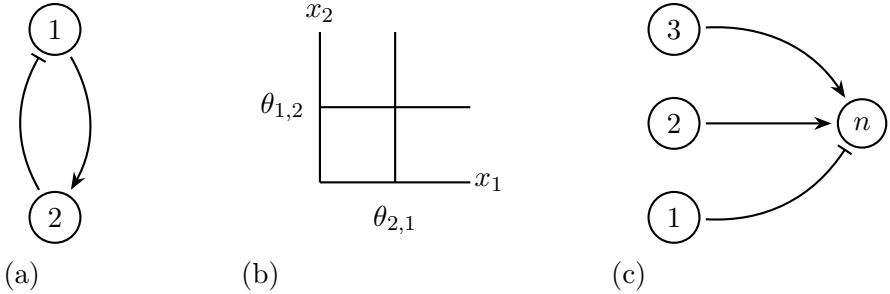
\begin{figure}[!htbp]
\centering
\begin{picture}(400,110)
\put(0,0){(a)}
\put(0,15){
\begin{tikzpicture}[thick, main node/.style={circle, fill=white, draw}, scale=2.5]

\node[main node] (1) at (0,1) {$1$};
\node[main node] (2) at (0,0) {$2$};

\draw[-{Stealth}, shorten <= 2pt, shorten >= 2pt] (1) to [bend left] (2);
\draw[-|, shorten <= 2pt, shorten >= 2pt] (2) to [bend left] (1);
\end{tikzpicture}
}
\put(90,0){(b)}
\put(90,15){
\begin{tikzpicture}[thick, scale=1.0]

\draw (0,0)--(2,0);
\draw (0,0)--(0,2);
\draw (1,0)--(1,2);
\draw (0,1)--(2,1);

\node at (-0.5,1) {$\theta_{1,2}$};
\node at (1,-0.5) {$\theta_{2,1}$};
\node at (0,2.25) {$x_2$};
\node at (2.25,0) {$x_1$};

\end{tikzpicture}
}
\put(230,0){(c)}
\put(240,15){
\begin{tikzpicture}[thick, main node/.style={circle, fill=white, draw}, scale=2.5]

\node[main node] (1) at (0,0) {$1$};
\node[main node] (2) at (0,0.5) {$2$};
\node[main node] (3) at (0,1) {$3$};
\node[main node] (4) at (1,0.5) {$n$};

\draw[-|, shorten <= 2pt, shorten >= 2pt] (1) to [bend right] (4);
\draw[-{Stealth}, shorten <= 2pt, shorten >= 2pt] (2) -- (4);
\draw[-{Stealth}, shorten <= 2pt, shorten >= 2pt] (3) to [bend left] (4);
\end{tikzpicture}
}
\end{picture}
\caption{(a) A 2-node network; (b) Decomposition of phase space of the network in (a) into 4 domains by thresholds; (c) Nodes 1, 2, and 3 affect node $n$ (see text for details).}
\label{fig:toggle}
\end{figure}

The functions $\lambda_{1,2}^-$ and $\lambda_{2,1}^+$ are constant off the hyperplanes $x_1 = \theta_{2,1}$ and $x_2 = \theta_{1,2}$, and thus there is a natural decomposition of phase space into rectangular regions (see Figure~\ref{fig:toggle}(b)).
We focus for the moment on the behavior of $x_1$, restricting our attention to whether within one of these regions $x_1$ is increasing or decreasing, i.e.\ on the sign of 
\begin{equation} \label{eq:sign}
-\gamma_1 x_1  + \begin{cases}
\ell_{1,2} + \delta_{1,2} & \text{if $x_2 < \theta_{1,2}$} \\
\ell_{1,2}& \text{if $x_2 > \theta_{1,2}$}
\end{cases}
\end{equation}

Note that since the first term is linear in $x_1$ it is sufficient to determine this at the threshold $x_1 = \theta_{2,1}$. Therefore, the answer depends entirely on the parameters, i.e.\  on one of the three possible relationships between parameters 
\begin{equation}
    \label{eq:ParameterDecomposition}
    \gamma_1\theta_{2,1} < \ell_{1,2} < \ell_{1,2} + \delta_{1,2}, \quad
    \ell_{1,2} < \gamma_1\theta_{2,1} < \ell_{1,2} + \delta_{1,2}, \quad
    \ell_{1,2} < \ell_{1,2} + \delta_{1,2} < \gamma_1\theta_{2,1} .
\end{equation}

\begin{rem}
\label{rem:OrderRelation}
As will be made clear shortly, it is useful to derive the relations of \eqref{eq:ParameterDecomposition} by beginning with the linear order $\ell_{1,2} < \ell_{1,2} + \delta_{1,2}$ and then considering all possible relative values of $\gamma_1\theta_{2,1}$.
\end{rem}

Observe that \eqref{eq:ParameterDecomposition} provides an explicit decomposition of $(0,\infty)^4$, the parameter space associated with node 1, that we codify as a graph
\begin{equation}
\label{eq:FactorGraph}
\begin{tikzpicture}[thick, main node/.style={rectangle, fill=white, draw, font=\small}, scale=1.0]

\node[main node] (1) at (0,0) {$\gamma_1\theta_{2,1} < \ell_{1,2} < \ell_{1,2} + \delta_{1,2}$};
\node[main node] (2) at (5,0) {$\ell_{1,2} < \gamma_1\theta_{2,1} < \ell_{1,2} + \delta_{1,2}$};
\node[main node] (3) at (10,0) {$\ell_{1,2} < \ell_{1,2} + \delta_{1,2} < \gamma_1\theta_{2,1}$};

\draw[-, shorten <= 2pt, shorten >= 2pt] (1) -- (2);
\draw[-, shorten <= 2pt, shorten >= 1pt] (2) -- (3);
\end{tikzpicture} 
\end{equation}
The nodes (represented by the rectangles) represent the regions in parameter space and edges indicate a single equality that defines adjacency of the regions.
We refer to this graph as the \emph{factor graph} associated to node 1 and denote it by $PG(1)$.
There is a similar factor graph for node 2 and the full \emph{parameter graph} is given by $PG := PG(1)\times PG(2)$.
Observe that this provides a decomposition of the full parameter space $(0,\infty)^8$ into 9 parameter domains.

We now consider a general regulatory network with  edges of the form of Figure~\ref{fig:RN+}(a).
The most significant difference is that a node may have multiple in-edges.
For the sake of notational simplicity assume that  node $n$ has $K$ in-edges coming from nodes $1,\ldots, K$, in which case by \eqref{eq:generalNonlinearity} we are interested in
\[
-\gamma_n x_n + \Lambda_n(x_1,\ldots,x_K).
\]
We make use of the following definition to express the allowable form of $\Lambda_n$.

\begin{defn}
\label{def:interaction_function}
\emph{Following \cite[Definition 1.1]{kepley:mischaikow:zhang}, an \emph{interaction function} of order $K$ is a polynomial in $K$ variables $z=(z_1,\ldots,z_K)$ of the form 
\[
f(z) := \prod_{j=1}^q f_j(z)
\]
where each factor has the form 
\[
f_j(z) = \sum_{i\in I_j}z_i
\]
and the indexing sets $\setof{I_j\mid 1\leq j\leq q}$ form a partition of $\setof{1,\ldots, K}$.
The interaction type of $f$ is $(k_1,\ldots, k_q) \in \mathbb{N}^q$
where $k_j$ denotes the number of elements of $I_j$ with $k_j \leq k_{j+1}$.}
\end{defn}

\begin{rem}
We denote the interaction type $(k_1, \ldots, k_q)$ with the convention that $k_j \leq k_{j+1}$ to match with the indexing used by DSGRN (\cite{kepley:mischaikow:zhang} uses the convention $k_j \geq k_{j+1}$).
\end{rem}

We assume that the  function $\Lambda_n$ is given by
\begin{equation}\label{eq:ln}
\Lambda_n(x_1,\ldots,x_K) = f \paren*{\lambda_{n,1}^{\pm}(x_1), \lambda_{n,2}^{\pm}(x_2), \dotsc, \lambda_{n, K}^{\pm}(x_K)},
\end{equation}
where $f$ is an interaction function of order $K$ and the sign of $\lambda_{n,j}^{\pm}$ is determined by the edge from node $j$ to node $n$.

In a slight abuse of notation, we sometimes also refer to the polynomial expression of the interaction function $f$ as its interaction type. As an example for the interaction function $f(z) = z_1(z_2+z_3)$ of type $(1,2)$, we denote this interaction type as $z_1(z_2+z_3)$. This notation conveys the same information in a less compressed format and will be useful when describing type I and II edges.

Consider the following example.
Let node $n$ have three in-edges as shown in Figure~\ref{fig:toggle}(c).
Given that $1 \dashv n$, $2\rightarrow n$ and $3\rightarrow n$, a biologically motivated choice for $\Lambda_n$ is
\begin{align*}
&\Lambda_n(x_1,x_2,x_3)  = \lambda_{n,1}^-(x_1) \left(\lambda_{n,2}^+(x_2) + \lambda_{n,3}^+(x_3) \right)\\
&= \left(
\begin{cases}
\ell_{n,1} + \delta_{n,1} & \text{if}~ x_1<\theta_{n,1} \\
\ell_{n,1} & \text{if}~ x_1 > \theta_{n,1}
\end{cases}
\right)
\left(
\begin{cases}
\ell_{n,2} & \text{if}~ x_2<\theta_{n,2} \\
\ell_{n,2} + \delta_{n,2} & \text{if}~ x_2 > \theta_{n,2}
\end{cases} +
\begin{cases}
\ell_{n,3} & \text{if}~ x_3< \theta_{n,3} \\
\ell_{n,3} + \delta_{n,3} & \text{if}~ x_3 > \theta_{n,3}
\end{cases}
\right)
\end{align*}
The associated interaction function of type $(k_1, k_2)=(1,2)$ is 
\[
f(z) = z_1(z_2+z_3).
\]

Observe that the values that $\Lambda_n$ can assume are given by the following eight polynomials in parameters $\ell$ and $\delta$

\begin{align*}
p_0 & = \ell_{n,1} (\ell_{n,2} + \ell_{n,3}) &
p_4 & = \ell_{n,1} (\ell_{n,2} + \ell_{n,3} + \delta_{n,3}) \\
p_1 & = (\ell_{n,1} + \delta_{n,1}) (\ell_{n,2} + \ell_{n,3}) &
p_5 & =  (\ell_{n,1} + \delta_{n,1}) (\ell_{n,2} + \ell_{n,3} + \delta_{n,3}) \label{eq:pi3node} \\
p_2 & = \ell_{n,1} (\ell_{n,2} + \delta_{n,2} + \ell_{n,3}) &
p_6 & =  \ell_{n,1} (\ell_{n,2} + \delta_{n,2} + \ell_{n,3} + \delta_{n,3}) \\
p_3 & = (\ell_{n,1} + \delta_{n,1}) (\ell_{n,2} + \delta_{n,2} + \ell_{n,3}) &
p_7 & = (\ell_{n,1} + \delta_{n,1}) (\ell_{n,2} + \delta_{n,2} + \ell_{n,3} + \delta_{n,3}).
\end{align*}

To put this into proper perspective we return to Remark~\ref{rem:OrderRelation}.
For $\Lambda_1$, arising in the case where node 1 has a single in-edge, we have the polynomials 
\[
p_0 = \ell_{1,2}\quad\text{and}\quad p_1= \ell_{1,2}+\delta_{1,2}.
\]
Observing that $p_0 < p_1$ is the only admissible linear order for the values of these polynomials allows us to determine the decomposition of parameter space.
The reader can check that the complete list of admissible linear orders associated with $\Lambda_n$ is exactly
\begin{equation*}
\begin{aligned}
(0, 1, 2, 3, 4, 5, 6, 7) \\ 
(0, 1, 2, 3, 4, 6, 5, 7) \\
(0, 1, 2, 4, 3, 5, 6, 7) \\
(0, 1, 2, 4, 3, 6, 5, 7) \\
(0, 4, 2, 6, 1, 5, 3, 7) 
\end{aligned}
\quad 
\begin{aligned}
 (0, 1, 2, 4, 6, 3, 5, 7) \\
 (0, 1, 4, 2, 5, 3, 6, 7) \\
 (0, 1, 4, 2, 5, 6, 3, 7) \\
 (0, 1, 4, 2, 6, 5, 3, 7) \\
 (0, 4, 1, 5, 2, 6, 3, 7)
\end{aligned}
\quad 
\begin{aligned}
 (0, 1, 4, 5, 2, 3, 6, 7) \\
 (0, 1, 4, 5, 2, 6, 3, 7) \\
 (0, 2, 1, 3, 4, 6, 5, 7) \\
 (0, 2, 1, 4, 3, 6, 5, 7) \\
 (0, 4, 2, 1, 6, 5, 3, 7)
\end{aligned}
\quad 
\begin{aligned}
 (0, 2, 1, 4, 6, 3, 5, 7) \\
 (0, 2, 4, 1, 6, 3, 5, 7) \\
 (0, 2, 4, 6, 1, 3, 5, 7) \\
 (0, 4, 1, 2, 5, 6, 3, 7) \\
  (0, 4, 1, 2, 6, 5, 3, 7)
\end{aligned}
 \end{equation*}
 where $(0, 1, 2, 3, 4, 5, 6, 7)$ corresponds to $p_0 < p_1< p_2< p_3< p_4< p_5< p_6< p_7$.
To determine the associated factor graph we consider all possible  values of $\setof{\gamma_n\theta_{m,n}}_m$ relative to each linear order where $m$ ranges over all out-edges of node $n$.
This produces the factor graph $PG(n)$.

As is indicated in \cite{cummins:gedeon:harker:mischaikow:mok} and \cite[Table 1]{kepley:mischaikow:zhang} the complete list of admissible linear orders have been determined for $\Lambda_n$ associated with interaction functions of type
$(k_1,\ldots, k_q)$ given by
\begin{align}
\label{eq:logic_list}
& (1) \nonumber \\
& (1,1),\; (2) \nonumber \\
& (1,1,1),\; (1,2),\; (3) \nonumber \\
& (1,1,1,1),\; (1,1,2),\; (2,2),\; (1,3),\; (4)  \\
& (1,1,1,1,1),\; (1,1,1,2),\;  (5) \nonumber \\
& (1,1,1,1,1,1),\;  (6) \nonumber
\end{align}

Thus, the DSGRN software can, in principle, take as input any network consisting of edges of type  $0$ in Figure~\ref{fig:RN+}(a) with the restriction that at any node $n$ the production function $\Lambda_n$ is given by one of the  interaction functions  of the type listed  above. Using pre-computed factor graphs $PG(n)$, DSGRN software constructs the parameter graph as the product $PG = \prod_{n=1}^N PG(n)$. 
In applications, the most serious restriction is that size of the factor graphs grows rapidly as a function of the number of variables, e.g.\ for a node with interaction function of type $(6)$ with one out-edge the factor graph has 89,414,640 elements.

A goal of this paper is to identify the factor graphs, and hence make the DSGRN software applicable, for networks that involve edges of type Figure~\ref{fig:RN+}(b)-(d).

\begin{rem}
There is an alternative  characterization of parameter nodes in a factor parameter graph  as collections of Boolean functions.  Any input polynomial to node $n$  can be represented as a Boolean string of length $k$ -- the number of inputs to node $n$. As an example consider node $n$ with  three in-edges as shown in Figure~\ref{fig:toggle}(c) and discussed above. Then each of the polynomials $p_0, \ldots, p_7$ can be represented as a Boolean string $(b_1, b_2,b_3)$, $b_i \in B = \{0,1\}$ by assigning $b_i=0$ if $p$ contains $\ell_{n,i}$ and $b_i = 1$ if $p$ contains  $\ell_{n,i} + \delta_{n,i}$. 
Since  for each  threshold  $\theta_{m,n}$ of node $n$, the  parameter node contains 
a linear  order 
\[
p_0 < \ldots < p_i < \theta_{m,n} <p_j < \ldots < p_7,
\]
we can associate to it a Boolean function $g_{m,n}: B^3 \to B$ by 
setting $g_{m,n}(p_s) = 0$ if and only if $p_s < \theta_{m,n}$.
Therefore a node of a factor parameter graph that corresponds to a node $n$ of a regulatory network with $k$ inputs and $m$ outputs corresponds to a collection of $m$ Boolean functions with $k$ inputs. For a more detailed description see~\cite{crawford-kahrl:cummins:gedeon:2021}.
\end{rem}

\subsection{State transition graph}

To close the review of DSGRN, we discuss how dynamics is associated to each parameter node.
This association is the same for new types of interactions presented in this paper. 
As indicated in (\ref{eq:sign}), for all parameters that belong to a domain in the parameter space represented by a  parameter node, each $x_n$ is either increasing or decreasing on the boundaries of domains bounded by 
by thresholds  $\theta_{j,n}$.
This leads to a construction of a state transition graph that represents the dynamics of the system. 

To simplify the explanation we assume that the gene regulatory network does not contain repressing self edges, as was assumed in the original DSGRN \cite{cummins:gedeon:harker:mischaikow:mok}.
This assures that the state transition graph, as we describe below, is well defined. However, this is not an essential limitation. In \cite{gameiro:gedeon:kepley:mischaikow} this problem was fully resolved by replacing the threshold corresponding to a repressing self-edge by two thresholds and thus extending the state transition graph. In this paper we apply the same method used in \cite{gameiro:gedeon:kepley:mischaikow} to replace the thresholds corresponding to all self-edges in the network (both repressing and activating) by two thresholds.

We also allow networks with multiple edges between pairs of nodes and use the DSGRN extension described in \cite{gameiro:mischaikow:zheleznyak:21} to construct the state transition graph in that case.

The collection of hyperplanes $\setof{x_m = \theta_{*,m}}$, $m=1,\ldots, N$, provides a cubical decomposition of the phase space $(0,\infty)^N$.
We refer to the associated $N$-dimensional cubes as \emph{domains}, and 
these domains define the vertices of the state transition graph $\cF$.
To define the edges of $\cF$, let $\kappa$ be a domain, with an $(N-1)$-dimensional face $\sigma$ that is a subset of the hyperplane $x_n=\theta_{m,n}$.
The sign of $-\gamma_n \theta_{m,n} + \Lambda_n(\kappa)$ 
determines whether $\sigma$ points into or away from $\kappa$.
Note that for all parameters belonging to a parameter node of $PG$ this sign is the same and thus the following construction produces the same $\cF$.
The edges of the state transition graph $\cF$ are now defined by the following two rules.
\begin{description}
\item[R1] If all the faces of $\kappa$ point into $\kappa$, then $\kappa$ has a self edge.
\item[R2] If $\sigma$ is an $(N-1)$-dimensional face of two domains $\kappa$ and $\kappa'$ and $\sigma$ points away from $\kappa'$ and into $\kappa$, then there is an edge from $\kappa'$ to $\kappa$.
\end{description}

Given a directed graph $\cF,$ the condensation graph $\cF^{\text{SCC}},$ can be identified in linear time  \cite{cormen:leiserson:rivest:stein}.
Recall that $\cF^{\text{SCC}}$ is a directed acyclic graph with one node for each strongly connected component (SCC) of $\cF$ and hence a poset.
We define a SCC to be \emph{nontrivial} if it contains at least one edge.
We define $(\sM(\cF),\leq)$, the \emph{Morse poset} of $\cF$, to be the subposet of $\cF^{\text{SCC}}$  consisting of the nontrivial SCCs. 
The Haase diagram for $\sM(\cF)$ is called the \emph{Morse graph}.

An obvious question about the DSGRN state transition graph dynamics, characterized by Morse graphs, is whether this characterization applies to network ODE models where smooth nonlinearities approximate the piecewise constant functions $\Lambda_n(x)$. The correspondence between individual solutions of such a smooth perturbed system and a system generated by ODEs with right-hand sides (\ref{eq:generalNonlinearity}) was studied by Ironi~\cite{Ironi2011}; the correspondence between the lattices of attractors for ODEs in $\R^2$ was established in \cite{gedeon:harker:kokubu:mischaikow:oka}. As is demonstrated in~\cite{duncan1} there is a close correspondence between equilibria and their stability for the smooth perturbed systems and the characterization provided by DSGRN. Finally, bifurcations of DSGRN equilibria within a class of ramp function perturbations has been examined in~\cite{duncan2}.

We conclude this section with a summary:
\begin{enumerate}
\item For a given regulatory network and a type of interaction between inputs at each node of the network, there is a finite decomposition of the parameter space, encoded as the parameter graph $PG$, such that all parameters in a single parameter node admit the same state transition graph STG.
\item The long term dynamics of the STG is represented by its Morse graph.
\item The parameter graph $PG$ with a Morse graph at each vertex encodes a finite representation of the dynamics of the regulatory network.
\end{enumerate}

\section{Extension of DSGRN}
\label{sec:extension}

In this section we consider  three generalizations of the parameter space decomposition (PSD) problem originally described in \cite{kepley:mischaikow:zhang} which we review briefly.

Suppose that node $n$ has $K$ input edges which are labeled $\setof*{1, \dotsc, K}$ and assume that the rate expression for $x_n$ is given by $-\gamma_n x_n + \Lambda_n(x)$, which we refer to as a {\em classical} rate expression. Assume that $f$ is an interaction function as defined in Definition \ref{def:interaction_function} with order $K$ and type $(k_1,\dotsc, k_q)$, and $\Lambda_n$ is given by the formula \eqref{eq:ln}.
Observe that 
the image of $\Lambda_n$ consists of  
$2^K$ values
\[
\cP := \setof*{f(z_1, \dotsc, z_K) \mid z_j \in \{\ell_{n, j}, \ell_{n, j} + \delta_{n,j}\}, \  1 \leq j \leq K}
\]
regardless of the sign of $\lambda_{n,j}^{\pm}$ for $1 \leq j \leq K$, and independent of the interaction type of $f$. The elements of $\cP$ are polynomial expressions in the $2K$ parameters $\setof{\ell_{n,j}, \delta_{n,j} \mid 1 \leq j \leq K}$. 

The \emph{PSD problem} associated with $f$ is to compute all possible linear orders of $\cP$ subject to the constraints $\ell_{n,j}, \delta_{n,j} > 0$ for all $1 \leq j \leq K$. 
In \cite{kepley:mischaikow:zhang} it is shown that given a solution of the PSD problem, the factor graph $PG(n)$ can be recovered with a trivial amount of post-processing.
In general, rigorously solving the PSD problem is difficult, but
it only needs be solved once for each interaction type and the results stored, see \eqref{eq:logic_list}.

In the remainder of this section we introduce the generalization of the PSD problem for the type I and  type II edges defined in Section~\ref{sec:intro}. These new rate expressions give rise to  new  PSD problems which we define below. Then, we demonstrate that the solutions of these new PSD problems can be obtained by creatively using   classical PSD solutions.

\subsection{Type I edges}
We start with the PSD for a rate expression that includes type I edges shown in Figure \ref{fig:RN+}(b). As is the case throughout this paper the mathematical expression governing the production of $x_n$ has the form \eqref{eq:generalNonlinearity} with $x = (x_1, \ldots, x_N) \in \R^N$.

However, to emphasize that we assume that a type I edge impacts  decay  as opposed to production we highlight the distinction of the variables and write
\begin{equation}
\label{eq:targeted_degredation_production_rate}
-\Gamma_n(\tilde{x})x_n + \Lambda_n(x)
\end{equation}
where $\Lambda_n$ has the form \eqref{eq:ln} and involves only type 0 edges, that is, it is defined by an interaction function $f$ of order $K$ and depends on the state variables $x_1, \ldots, x_K$. We will refer to $\Lambda_n$ as the \emph{classical production rate expression}. The function $\Gamma_n$ also has the form \eqref{eq:ln}, but involving only type I edges, that is, it is an interaction function composed of step functions that depend on the state variables $\tilde{x}_j \in \{ x_1, \ldots, x_N \}$ for $j = 1, \ldots, \tilde{K}$,
where the state variables $\{ x_1, \ldots, x_K \}$ and $\{ \tilde{x}_1, \dotsc, \tilde{x}_{\tilde{K}} \}$ form disjoint sets. Specifically,
\[
\Gamma_n(\tilde{x}) = \tilde{f} \paren*{\tilde{\lambda}_{n,1}^{\pm}(\tilde{x}_1), \tilde{\lambda}_{n,2}^{\pm}(\tilde{x}_2), \dotsc, \tilde{\lambda}_{n,\tilde{K}}^{\pm}(\tilde{x}_{\tilde{K}})}
\]
where $\tilde{f}$ is an interaction function of order $\tilde{K}$, and interaction type $(\tilde{k}_1, \dotsc, \tilde{k}_{\tilde{q}})$. For $1 \leq j \leq \tilde{K}$, $\tilde{\lambda}_{n,j}^{\pm}$ is a step function depending on parameters $\{ \tilde{\ell}_{n,j}, \tilde{\delta}_{n,j} \}$ and  the sign of $\tilde{\lambda}_{n, j}^{\pm}$ is determined by the type of edge in Figure \ref{fig:RN+}(b). By a similar observation as in the classical rate expression 
we note that $\Gamma_n$ is a simple function which takes the values
\[
\tilde{\cP} := \setof*{\tilde{f} \paren*{z_1, \dotsc, z_{\tilde{K}}} \mid z_j \in \setof*{\tilde{\ell}_{n, j}, \tilde{\ell}_{n,j} + \tilde{\delta}_{n, j}}, \ 1 \leq j \leq \tilde{K}}
\]
regardless of signs of each step function or the interaction type of $\tilde{f}$.

Analogous to the classical rate expression, we are interested in determining the sign of a \textit{type I rate expression} of the form

\begin{equation}
\label{eq:typeI_sign_expression}
-\Gamma_n(\tilde{x}) \theta_{*, n} + \Lambda_n(x)
\end{equation}
where  $\theta_{*, n}$ is an arbitrary threshold associated to an outgoing edge whose source is node $n$. 

A crucial observation is that the sign determination of \eqref{eq:typeI_sign_expression} depends only on the value of $\theta_{*, n}$ relative to the expression $\Lambda_n(x) / \Gamma_n(\tilde{x})$.
The problem of sign determination for \eqref{eq:typeI_sign_expression} over all parameter values can be  reduced to the problem of determining all possible linear orders of the finite set
\begin{equation}
\cR := \setof*{\frac{p}{\tilde{p}} \mid p \in \cP, \ \tilde{p} \in  \tilde{\cP}},
\end{equation}
subject to the constraints $\ell_{n,i}, \delta_{n,i} > 0$ for $1 \leq i \leq K$ and $\tilde{\ell}_{n,j}, \tilde{\delta}_{n,j} > 0$ for $1 \leq j \leq \tilde{K}$.
This is the \emph{joint PSD problem} associated to the pair $(\tilde{f}, f)$ with a {\em joint interaction type} denoted by $(\tilde{k}_1, \dotsc, \tilde{k}_{\tilde{q}}; k_1, \dotsc, k_q)$. As in the classical case, given a solution of the joint PSD problem the sign of Equation \eqref{eq:typeI_sign_expression} can be determined for the entire parameter space with trivial post-processing.

Similarly to the case of a type 0 rate expression, we use a shorthand notation to describe the interaction type of a type I rate expression. In this case we put the polynomial expression of $\tilde{f}$ within $\langle \cdot \rangle$ and add to that the polynomial expression of $f$, that is, we denote the interaction type by
$\langle \tilde{f} \rangle + f$. As an example, assuming we are using a rate expression with interaction type $(1; 2)$ for node $n$ in Figure~\ref{fig:phosphorylation_edges}(a), we express this fact by saying that the interaction type of node $n$ is
\[
\langle x \rangle + y + z.
\]

Before presenting the general result, we demonstrate how to obtain the linear orders of a type I rate expression in a simple example. 

\subsection{Example}
Consider the simplest possible type I rate expression for $x_n$, 
\[
-\Gamma_n(x_1) x_n + \Lambda_n(x_2)
\]
where $x_1$ and $x_2$ are state variables. For simplicity, we suppress the dependence on $n$ which usually appears in parameter subscripts and assume that $x_1$ and $x_2$ are both promoters of $x_n$ so that we have the formulas  
\[
\Gamma_n(x_1) = 
\begin{cases}
\ell_1 + \delta_1 & \text{if} \ x_1 < \theta_1 \\
\ell_1 & \text{if} \ x_1 > \theta_1 \\
\end{cases}
\qquad 
\Lambda_n(x_2) =
\begin{cases}
\ell_2 & \text{if} \ x_2 < \theta_2 \\
\ell_2 + \delta_2 & \text{if} \ x_2 > \theta_2 \\
\end{cases}
\]
where $\setof*{\ell_1, \delta_1, \theta_1, \ell_2, \delta_2, \theta_2}$ are positive parameters. The interaction type associated with this example is $(1; 1)$ and it is denoted by $\langle z_1 \rangle + z_2$.

The associated joint PSD problem is to determine all admissible linear orders of the following rational expressions
\[
\cR := \setof*{\frac{\ell_2}{\ell_1}, \frac{\ell_2 + \delta_2}{\ell_1}, \frac{\ell_2}{\ell_1 + \delta_1}, \frac{\ell_2 + \delta_2}{\ell_1 + \delta_1}}
\]
subject to the constraints that all parameters must be strictly positive. 

Define the polynomials
\[
p_0 = \ell_2, \quad p_1 = \ell_2 + \delta_2, \quad \tilde{p}_0 = \ell_1, \quad \tilde{p}_1 = \ell_1 + \delta_1
\]
which we collect into distinct subsets $\cP := \setof*{p_0, p_1}$ and $\tilde{\cP} := \setof*{\tilde{p}_0, \tilde{p}_1}$. We let the set $\cR$ inherit this indexing by writing $\cR = \setof*{r_{ij} \mid r_{ij} = \frac{p_i}{\tilde{p}_j}, \  0 \leq i, j, \leq 1}$. Let $\mu = (\ell_1, \delta_1, \ell_2, \delta_2)$ denote an arbitrary parameter vector and observe that for any $\mu \in (0, \infty)^4$ we have
\[
r_{01}(\mu) < r_{00}(\mu) < r_{10}(\mu) \qquad \text{and} \quad r_{01}(\mu) < r_{11}(\mu)< r_{10}(\mu)
\]
which induces a partial order on $\cR$ given by $r_{01} \prec \setof*{r_{00}, r_{11}} \prec r_{10}$. One easily checks that $\mu$ can be chosen such that either linear order of the set $\setof*{r_{00}(\mu), r_{11}(\mu)}$ is possible. Thus, there are exactly two admissible linear extensions of $(\cR, \prec)$ given by 
\begin{equation}
\label{eq:example_<x>+y_orders}
\paren*{r_{01}, r_{00}, r_{11}, r_{10}}, \qquad \paren*{r_{01}, r_{11}, r_{00}, r_{10}}.
\end{equation}

We compare this result to a PSD problem for interaction type $(1,1)$ which corresponds to a classical rate expression of the form 
\[
-\gamma_n x_n + \Lambda_n(x_1, x_2)
\]
where $\gamma_n$ is a positive parameter and $\Lambda_n$ is given by the formula
\[
\Lambda_n(x_1, x_2) = \lambda^+_1(x_1) \lambda^+_2(x_2). 
\]
The classical PSD is to determine all possible linear orders of the polynomials
\[
\cQ := \setof*{\ell_2 \ell_1, \ell_2 (\ell_1 + \delta_1), (\ell_2 + \delta_2) \ell_1, (\ell_2 + \delta_2) (\ell_1 + \delta_1)} = \setof*{p_0 \tilde{p}_0, p_0 \tilde{p}_1, p_1 \tilde{p}_0, p_1 \tilde{p}_1}. 
\]
In the previous version of DSGRN these 4 polynomials are endowed with a partial order by assigning labels via the indexing
\[
q_0 := \ell_1 \ell_2, \quad 
q_1 := (\ell_1 + \delta_1) \ell_2, \quad
q_2 := \ell_1 (\ell_2 + \delta_2), \quad
q_3 := (\ell_1 + \delta_1)(\ell_2 + \delta_2)
\]
with the partial order $q_0 \prec \setof*{q_1, q_2} \prec q_3$. One easily checks that both possible orders between $q_1$ and $q_2$ can be satisfied. Thus, the solution to this classical PSD problem is the two possible linear orders
\begin{equation}
\label{eq:example_xy_orders}
(0, 1, 2, 3), \qquad (0, 2, 1, 3),
\end{equation}
which is stored in DSGRN. However, if instead we indexed the same 4 polynomials as follows 
\[
\bar{q}_0 := p_0 \tilde{p}_0, \quad 
\bar{q}_1 := p_1 \tilde{p}_0, \quad
\bar{q}_2 := p_0 \tilde{p}_1, \quad 
\bar{q}_3 := p_1 \tilde{p}_1,
\]
then we observe that 
\[
\bar{q}_1(\mu) < \bar{q}_2(\mu) \iff \frac{p_1(\mu)}{\tilde{p}_1(\mu)} < \frac{p_0(\mu)}{\tilde{p}_0(\mu)} \iff r_{11}(\mu) < r_{00}(\mu).
\]
It follows that finding all admissible linear orders of $\cR$ is equivalent to finding all admissible linear orders of $\cQ$ with respect to the new indexing. However, this has already been computed and stored in DSGRN and we only require the correct change of indexing on these polynomials in order to reuse it and this is trivial to store in a table and lookup as needed. 

Intuitively, choosing the relative order of the two ``free'' polynomials for the $(1,1)$ PSD problem is equivalent to choosing the relative order of the two ``free'' rational expressions for the $(1;1)$ PSD problem. Moreover, given the two possible linear orders for $\cQ$, the two admissible orders for $\cR$ are immediately recovered.

Next, we prove that solving any joint PSD problem is equivalent to solving a related classical PSD problem. In particular, the joint PSD problem can be solved using the algorithms described in \cite{kepley:mischaikow:zhang}. 

\begin{thm}
\label{thm:type_I_bijection}
Suppose $x_n$ is governed by a type I rate expression $-\Gamma_n(\tilde{x}) x_n + \Lambda_n(x)$ with associated interaction functions $\tilde{f}$ and $f$ of orders $\tilde{K}$ and $K$ and interaction types $(\tilde{k}_1, \dotsc, \tilde{k}_{\tilde{q}})$ and $(k_1, \dotsc, k_q)$ respectively, with $\{ x_1, \ldots, x_K \}$ and $\{ \tilde{x}_1, \dotsc, \tilde{x}_{\tilde{K}} \}$ disjoint sets of state variables.

Then, the joint PSD problem associated to $(\tilde{f}, f)$ is equivalent to the classical PSD problem associated to the interaction function $f \cdot \tilde{f}$ in the following sense.
\begin{enumerate}
\item There is bijection between solutions of the joint PSD problem for $(\tilde{f}, f)$ and solutions of the classical PSD problem associated with $f \cdot \tilde{f}$. 
\item This bijection can be efficiently and explicitly constructed i.e.~the admissible linear orders for the joint PSD problem can be efficiently recovered given the admissible linear orders for the classical PSD problem.
\end{enumerate}
\end{thm}

\begin{proof}
Let $g = f \cdot \tilde{f}$ and observe that $g$ is an interaction function of order $K + \tilde{K}$ and interaction type $(k_1, \dotsc, k_q, \tilde{k}_1, \dotsc, \tilde{k}_{\tilde{q}})$.

Observe that this interaction type need not satisfy our previously stated convention that the summand sizes are in increasing order. For instance, it is possible that $\tilde{k}_1 < k_q$. However, this convention is simply a notational convenience since any interaction function is invariant under permutation of the summands.
In particular, the set of admissible linear orders for an interaction type depends only on the summand sizes and does not depend on their order in the vector defining the interaction type.

The PSD problem associated with $g$ is to determine all possible linear orders for a collection of $2^{K + \tilde{K}}$ polynomials denoted by
\[
\cQ \subset \rr \brak*{\ell_{n,1}, \dotsc, \ell_{n,K}, \delta_{n,1}, \dotsc, \delta_{n,K}, \tilde{\ell}_{n, 1}, \dotsc, \tilde{\ell}_{n, \tilde{K}}, \tilde{\delta}_{n, 1}, \dotsc, \tilde{\delta}_{n, \tilde{K}}}, 
\]
subject to the constraint that each of the $2(K + \tilde{K})$ indeterminates is positive. Observe that from the construction of $g$ and the definition of the classical PSD problem, each $q \in \cQ$ admits a unique factorization of the form $q = p \cdot \tilde{p}$ where
\begin{align*}
p \in \cP & := \setof*{f \paren*{z_1, \dotsc, z_{K}} \mid z_i \in \setof*{\ell_{n, i}, \ell_{n,i} + \delta_{n, i}}, \ 1 \leq i \leq K} \\
\tilde{p} \in \tilde{\cP} & := \setof*{\tilde{f} \paren*{z_1, \dotsc, z_{\tilde{K}}} \mid z_j \in \setof*{\tilde{\ell}_{n, j}, \tilde{\ell}_{n,j} + \tilde{\delta}_{n, j}}, \ 1 \leq j \leq \tilde{K}}.
\end{align*}
We note that the collection of polynomials defining $\cQ$ is closely related to the rational expressions appearing in the PSD problem associated to $(\tilde{f}, f)$. Specifically, the joint PSD problem associated to ($f, \tilde{f})$ is to determine all possible linear orders for the collection of $2^{K + \tilde{K}}$ rational functions defined over the same $2(K + \tilde{K})$ indeterminate parameters denoted by
\begin{equation}
\label{eq:rational_PSD}
\cR := \setof*{\frac{p}{\tilde{p}} \mid p_i \in \cP, \  \tilde{p}_j \in \tilde{\cP}} \subset \rr \brak*{\ell_{n,1}, \dotsc, \delta_{n,K}, \tilde{\ell}_{n, 1}, \dotsc, \tilde{\delta}_{n, \tilde{K}}}, 
\end{equation}
subject to the same positivity constraints on the parameters. 

In order to establish the bijection between linear orders of $\cQ$ and $\cR$ we note the following useful fact. Suppose $(\xi, \tilde{\xi}) \in (0, \infty)^{2K} \times (0, \infty)^{2\tilde{K}}$ where
\[
\xi := \paren*{\ell_{n,1}, \dotsc, \delta_{n,K}} \quad \text{and} \quad 
\tilde{\xi} := \paren*{\tilde{\ell}_{n, 1}, \dotsc, \tilde{\delta}_{n, \tilde{K}}}.
\]
Then for any $p_1, p_2 \in \cP$ and $\tilde{p}_1, \tilde{p}_2 \in \tilde{\cP}$ we have the equivalence,  
\begin{equation}
\label{eq:R_Q_inequality_equivalence}
\frac{p_1(\xi)}{\tilde{p}_1(\tilde{\xi})} < \frac{p_2(\xi)}{\tilde{p}_2(\tilde{\xi})} \quad \text{ if and only if }
\quad p_1(\xi) \cdot \tilde{p}_2(\tilde{\xi}) < \tilde{p}_1(\tilde{\xi}) \cdot p_2(\xi). 
\end{equation}
The proof of this claim is a trivial consequence of the fact that elements of $\cP, \tilde{\cP}$ have only positive coefficients due to Definition \ref{def:interaction_function} and all coordinates of $\xi, \tilde{\xi}$ are strictly positive. Therefore, the quantities $p_1(\xi), \tilde{p}_1(\tilde{\xi}), p_2(\xi), \tilde{p}_2(\tilde{\xi})$ are strictly positive for any choices of $\xi \in (0, \infty)^{2K}$ and $\tilde{\xi} \in (0, \infty)^{2\tilde{K}}$ and the equivalence claimed in \eqref{eq:R_Q_inequality_equivalence} follows.

To complete the proof we assume that $\cP$ and $\tilde{\cP}$ are indexed by unspecified but fixed indexing maps onto the integers $I := \setof*{0, \dotsc, 2^K-1}$ and $J := \setof*{0, \dotsc, 2^{\tilde{K}}-1}$ respectively. We use subscripts to denote these indices so that the sets $\cP$ and $\tilde{\cP}$ can be expressed as 
\[
\cP = \setof*{p_i \mid i \in I}, \qquad \tilde{\cP} = \setof*{\tilde{p}_j \mid j \in J}.
\]
These indices induce associated indexing maps on $\cQ$ and $\cR$ defined by 
\[
q_{ij} := p_i \cdot \tilde{p}_j \in \cQ, \quad r_{ij} := \frac{p_i}{\tilde{p}_j} \in \cR,
\]
for any $i \in I, j \in J$. Observe that with these indices for $\cQ$ and $\cR$ fixed, any linear order of $\cQ$ or $\cR$ can be uniquely identified with a permutation on $2^{K + \tilde{K}}$ symbols i.e.~an element of $S_{2^{K+\tilde{K}}}$. Consequently, the solution to the PSD problems associated with either $g$ or $(\tilde{f}, f)$ can be identified with subsets of $S_{2^{K+\tilde{K}}}$ denoted by $\cT_g$ and $ \cT_{(f,\tilde{f})}$ respectively. 

Let us consider $\sigma \in \cT_{g}$ which defines a linear order on $\cQ$ which we interpret as a function, $\sigma : \setof*{0, \dotsc, 2^{K + \tilde{K}}-1} \to I \times J$.
The assumption that $\sigma$ defines an admissible linear order implies there exists a witness 
\[
\mu = \paren*{\ell_{n,1}, \dotsc, \ell_{n,K}, \delta_{n,1}, \dotsc, \delta_{n,K}, \tilde{\ell}_{n, 1}, \dotsc, \tilde{\ell}_{n, \tilde{K}}, \tilde{\delta}_{n, 1}, \dotsc, \tilde{\delta}_{n, \tilde{K}}} \in (0, \infty)^{2(K + \tilde{K})}
\]
such that
\[
q_{\sigma(0)}(\mu) < q_{\sigma(1)}(\mu) < \dots <  q_{\sigma(2^{K + \tilde{K}}-1)}(\mu),
\]
or equivalently, 
\[
q_{\sigma(m-1)}(\mu) < q_{\sigma(m)}(\mu) \qquad \forall 1 \leq m \leq 2^{K+\tilde{K}}-1. 
\]
Fix an arbitrary $m \in \setof*{1, \dotsc, 2^{K+\tilde{K}}-1}$ and assume that $\sigma(m-1) = (i,j)$ and $\sigma(m) = (i', j')$. We split $\mu$ as 
\[
\mu = (\xi, \tilde{\xi}) \in (0, \infty)^{2K} \times (0, \infty)^{2\tilde{K}}
\] 
and recalling that $\cP, \tilde{\cP}$ are polynomials over disjoint parameters, we have
\begin{equation}
\label{eq:sigma_IJ_ineq}
q_{\sigma(m-1)}(\mu) = p_i(\xi) \cdot \tilde{p}_j(\tilde{\xi}) <  p_{i'}(\xi) \cdot \tilde{p}_{j'}(\tilde{\xi}) = q_{\sigma(m)}(\mu)
\end{equation}
and applying \eqref{eq:R_Q_inequality_equivalence} we conclude that
\begin{equation}
\label{eq:tau_IJ_ineq}
r_{i j'}(\mu) = \frac{p_{i}(\xi)}{\tilde{p}_{j'}(\tilde{\xi})} < \frac{p_{i'}(\xi)}{\tilde{p}_{j}(\tilde{\xi})} = r_{i' j}(\mu).
\end{equation}

Applying the same argument for each $1 \leq m \leq 2^{K + \tilde{K}} - 1$ yields a distinct linear order on the elements of $\cR$. Observe that this is not the same linear order as $\sigma$, but it is induced directly by $\sigma$ and the same $\mu$ is a witness for both. In other words, the existence of $\sigma \in \cT_g$ implies existence of a related linear order $\tau \in \cT_{(\tilde{f}, f)}$ such that 
\[
r_{\tau(0)}(\mu) < r_{\tau(1)}(\mu) < \dots <  r_{\tau(2^{K + \tilde{K}}-1)}(\mu),
\]
where $\mu$ is any witness for $\sigma$. It follows that $\# \cT_g \leq \# \cT_{(\tilde{f}, f)}$. However, a similar argument shows that any linear order $\tau \in \cT_{(\tilde{f}, f)}$ yields a distinct admissible linear order on $\cQ$ by applying the converse of \eqref{eq:R_Q_inequality_equivalence} to successive pairs of elements ordered by $\tau$ which completes the proof of the first claim. 

To prove the second claim, we simply observe that if $\cT_g$ has been computed and stored, then $\cT_{(\tilde{f}, f)}$ is recovered by a small number of trivial lookup operations. Specifically, for each $\sigma \in \cT_g$, the inverse indexing map for $\cQ$ which decomposes $q_{\sigma(m)}$ into factors as in Equation \eqref{eq:sigma_IJ_ineq} must be evaluated, followed by an evaluation of the indexing map for $r_{\tau(m)}$ as in Equation \eqref{eq:tau_IJ_ineq}. This must be done for each $1 \leq m \leq 2^{K + \tilde{K}}$, in order to completely construct $\tau \in \cT_{(\tilde{f}, f)}$ from a given linear order $\sigma \in \cT_g$. 
\end{proof}

Theorem~\ref{thm:type_I_bijection} demonstrates that solving the PSD problem for type I rate expressions is equivalent to solving the PSD problem for a related classical rate expression. In particular, for many network topologies of interest these classical PSD problems have already been solved and implemented in the current version of DSGRN implying that the implementation of type I rate expressions often requires only minor modifications to the existing DSGRN library. See Table~\ref{table:ubiquitination_logic} for a list of all interaction types available.

\begin{table}[!htbp]
\centering
\renewcommand{\arraystretch}{1.2}
\begin{tabular}{@{}lll@{}}
\toprule
\midrule
PTM interaction type & DSGRN interaction type & Interaction type \\
\midrule
$\langle x \rangle + y$ & $xy$ & $(1 ; 1)$ \\
\midrule
$\langle x \rangle + yz$ & $xyz$ & $(1 ; 1,1)$ \\
$\langle xy \rangle + z$ & $xyz$ & $(1,1 ; 1)$ \\
\midrule
$\langle x \rangle + y + z$ & $x(y+z)$ & $(1 ; 2)$ \\
$\langle y+z \rangle + x$ & $x(y+z)$ & $(2 ; 1)$ \\
\midrule
$\langle x \rangle + yzw$ & $xyzw$ & $(1 ; 1,1,1)$ \\
$\langle xy \rangle + zw$ & $xyzw$ & $(1,1 ; 1,1)$ \\
$\langle xyz \rangle + w$ & $xyzw$ & $(1,1,1 ; 1)$ \\
\midrule
$\langle x \rangle + y(z+w)$ & $xy(z+w)$ & $(1 ; 1,2)$ \\
$\langle xy \rangle + z+w$ & $xy(z+w)$ & $(1,1 ; 2)$ \\
$\langle z+w \rangle + xy$ & $xy(z+w)$ & $(2 ; 1,1)$ \\
$\langle x(z+w) \rangle + y$ & $xy(z+w)$ & $(1,2 ; 1)$ \\
\midrule
$\langle x+y \rangle + z+w$ & $(x+y)(z+w)$ & $(2 ; 2)$ \\
\midrule
$\langle x \rangle + y+z+w$ & $x (y+z+w)$ & $(1 ; 3)$ \\
$\langle y+z+w \rangle + x$ & $x (y+z+w)$ & $(3 ; 1)$ \\
\midrule
$\langle x \rangle + yzuw$ & $xyzuw$ & $(1 ; 1,1,1,1)$ \\
$\langle xy \rangle + zuw$ & $xyzuw$ & $(1,1 ; 1,1,1)$ \\
$\langle xyz \rangle + uw$ & $xyzuw$ & $(1,1,1 ; 1,1)$ \\
$\langle xyzu \rangle + w$ & $xyzuw$ & $(1,1,1,1 ; 1)$ \\
\midrule
$\langle x \rangle + yz(u + w)$ & $xyz(u + w)$ & $(1 ; 1,1,2)$ \\
$\langle xy \rangle + z(u + w)$ & $xyz(u + w)$ & $(1,1 ; 1,2)$ \\
$\langle xyz \rangle + u + w$ & $xyz(u + w)$ & $(1,1,1 ; 2)$ \\
$\langle u + w \rangle + xyz$ & $xyz(u + w)$ & $(2 ; 1,1,1)$ \\
$\langle x(u + w) \rangle + yz$ & $xyz(u + w)$ & $(1,2 ; 1,1)$ \\
$\langle xy(u + w) \rangle + z$ & $xyz(u + w)$ & $(1,1,2 ; 1)$ \\
\midrule
$\langle x \rangle + yzuvw$ & $xyzuvw$ & $(1 ; 1,1,1,1,1)$ \\
$\langle xy \rangle + zuvw$ & $xyzuvw$ & $(1,1 ; 1,1,1,1)$ \\
$\langle xyz \rangle + uvw$ & $xyzuvw$ & $(1,1,1 ; 1,1,1)$ \\
$\langle xyzu \rangle + vw$ & $xyzuvw$ & $(1,1,1,1 ; 1,1)$ \\
$\langle xyzuv \rangle + w$ & $xyzuvw$ & $(1,1,1,1,1 ; 1)$ \\
\bottomrule
\end{tabular}
\caption{Complete list of available type I interaction types involving type I and type 0 edges (first column). The size of the parameter factor graph for each of the interaction types is the same as the size of the corresponding original DSGRN interaction types listed (middle column).}
\label{table:ubiquitination_logic}
\end{table}

\subsection{Type II edges}
\label{sec:typeIIedges}

Type II edges are pairs of edges where one edge connects a node to the other edge
as shown in Figure~\ref{fig:RN+}(c)-(d), however they are implemented as pairs of edges from the source node to the target node. For example, the pair of type II edges in Figure~\ref{fig:phosphorylation_edges}(b) is implemented as the pair of edges from nodes 1 and 2 to node 3 as shown in Figure~\ref{fig:phosphorylation_edges}(c) and the type II edges in Figure~\ref{fig:phosphorylation_edges}(d) are implemented as depicted in Figure~\ref{fig:phosphorylation_edges}(e).

The mathematical expression governing the production of $x_n$ that includes type II edges is called a \emph{type II rate expression} and has the form
\begin{equation}
\label{eq:type_II_rate_expression}
-\gamma_n x_n + \Lambda_n(\tilde{w}_1, \dotsc, \tilde{w}_{\tilde{K}}, x_1, \dotsc, x_K)
\end{equation}
where $x_1, \dotsc, x_K$ represent nodes $1, \ldots, k$ connected to node $n$ by type $0$ edges and $\tilde{w}_1, \dotsc, \tilde{w}_{\tilde{K}}$ represent pairs of nodes connected to $n$ by type II edges, that is, $\tilde{w}_k = (\tilde{x}_{i_k}, \tilde{y}_{j_k})$, with $\tilde{x}_{i_k}, \tilde{y}_{j_k} \in \{ x_1, \ldots, x_N \}$, represents a pair of type II edges for $k = 1, \dotsc, \tilde{K}$. We assume that $\{ x_1, \ldots, x_K \}$, $\{ \tilde{x}_{i_1}, \ldots, \tilde{x}_{i_{\tilde{K}}} \}$, and $\{ \tilde{y}_{i_1}, \ldots, \tilde{y}_{i_{\tilde{K}}} \}$ are disjoint sets.

The production rate function $\Lambda_n$ is given in terms of an interaction function $f$ as defined in Definition~\ref{def:interaction_function} as follows.

In contrast to type 0 and I edges, to a pair of type II edges we associate a single $\ell$ and $\delta$. More precisely, if $(\tilde{x}_{i_k}, \tilde{y}_{j_k})$ represents a pair of type II edges terminating at the node $n$, to this pair we associate the positive parameters $\tilde{\ell}_{n, (i_k, j_k)}$, $\tilde{\delta}_{n, (i_k, j_k)}$. We still, however, associate one threshold per edge, that is, we associate $\tilde{\theta}_{n, i_k}$ and $\tilde{\theta}_{n, j_k}$ to the edges from $i_k$ and $j_k$ to $n$, respectively.

To simplify the presentation we denote
\[
[\tilde{x}_{i_k}, \tilde{y}_{j_k}] := \min \{ \lambda_{n, i_k}^{\pm}(\tilde{x}_{i_k}; \tilde{\ell}_{n, (i_k, j_k)}, \tilde{\delta}_{n, (i_k, j_k)}, \tilde{\theta}_{n, i_k}), \lambda_{n, j_k}^{\pm}(\tilde{y}_{j_k}; \tilde{\ell}_{n, (i_k, j_k)}, \tilde{\delta}_{n, (i_k, j_k)}, \tilde{\theta}_{n, j_k}) \}
\]
where $\lambda_{n, i_k}^{\pm}(\tilde{x}_{i_k}; \tilde{\ell}_{n, (i_k, j_k)}, \tilde{\delta}_{n, (i_k, j_k)}, \tilde{\theta}_{n, i_k})$ and $\lambda_{n, j_k}^{\pm}(\tilde{y}_{j_k}; \tilde{\ell}_{n, (i_k, j_k)}, \tilde{\delta}_{n, (i_k, j_k)}, \tilde{\theta}_{n, j_k})$ are regular step functions (see \eqref{eq:lambdaplusminus}) corresponding to the edges from $i_k$ and $j_k$ to $n$ respectively.
When convenient we also use the notation $ [\lambda_{n, i_k}^{\pm}(\tilde{x}_{i_k}), \lambda_{n, j_k}^{\pm}(\tilde{y}_{j_k})] := [\tilde{x}_{i_k}, \tilde{y}_{j_k}]$.

Note that we included the parameters to emphasize that $\tilde{\ell}$ and $\tilde{\delta}$ are the same between the functions $\lambda_{n, i_k}^{\pm}$ and $\lambda_{n, j_k}^{\pm}$. As a consequence, both have values in the set $\{ \tilde{\ell}_{n, (i_k, j_k)}, \tilde{\ell}_{n, (i_k, j_k)} + \tilde{\delta}_{n, (i_k, j_k)} \}$, and therefore
\begin{equation}
\label{eq:tilde}
[\tilde{x}_{i_k}, \tilde{y}_{j_k}] =
\begin{cases}
\tilde{\ell}_{n, (i_k, j_k)}, & \text{if}~~ \lambda_{n, i_k}^{\pm}(\tilde{x}_{i_k}) = \tilde{\ell}_{n, (i_k, j_k)}
~\text{or}~ \lambda_{n, j_k}^{\pm}(\tilde{y}_{j_k}) = \tilde{\ell}_{n, (i_k, j_k)} \\
\tilde{\ell}_{n, (i_k, j_k)} + \tilde{\delta}_{n, (i_k, j_k)}, & \text{otherwise}.
\end{cases}
\end{equation}
Therefore $[\tilde{x}_{i_k}, \tilde{y}_{j_k}]$ acts like an \textsf{AND} operation on the values of $\lambda_{n, i_k}^{\pm}$ and $\lambda_{n, j_k}^{\pm}$, that is, $[\tilde{x}_{i_k}, \tilde{y}_{j_k}]$ is high if and only if both of those values are high.

The signs are determined by the type of edges as follows. Let $(\tilde{x}_{i_k}, \tilde{y}_{j_k})$ represent a pair of type II edges. If the pair of edges interact as in Figure~\ref{fig:RN+}(c) denoted by solid dot then
\begin{equation}
\label{eq:fulldot}
\lambda_{n, i_k}^{\pm}(\tilde{x}_{i_k}) =
\begin{cases}
\lambda_{n, i_k}^{+}(\tilde{x}_{i_k}), & \text{if}~ i_k \to n \\
\lambda_{n, i_k}^{-}(\tilde{x}_{i_k}), & \text{if}~ i_k \dashv n
\end{cases}
\quad \text{and} \quad
\lambda_{n, j_k}^{\pm}(\tilde{y}_{j_k}) =
\begin{cases}
\lambda_{n, j_k}^{+}(\tilde{y}_{j_k}), & \text{if}~ j_k \to n \\
\lambda_{n, j_k}^{-}(\tilde{y}_{j_k}), & \text{if}~ j_k \dashv n .
\end{cases}
\end{equation}
If the pair of edges interact as in Figure~\ref{fig:RN+}(d) denoted by an open dot then the sign of the second function is flipped, that is,
\begin{equation}
\label{eq:emptydot}
\lambda_{n, i_k}^{\pm}(\tilde{x}_{i_k}) =
\begin{cases}
\lambda_{n, i_k}^{+}(\tilde{x}_{i_k}), & \text{if}~ i_k \to n \\
\lambda_{n, i_k}^{-}(\tilde{x}_{i_k}), & \text{if}~ i_k \dashv n
\end{cases}
\quad \text{and} \quad
\lambda_{n, j_k}^{\pm}(\tilde{y}_{j_k}) =
\begin{cases}
\lambda_{n, j_k}^{-}(\tilde{y}_{j_k}), & \text{if}~ j_k \to n \\
\lambda_{n, j_k}^{+}(\tilde{y}_{j_k}), & \text{if}~ j_k \dashv n .
\end{cases}
\end{equation}

Let $f(\tilde{z}_1, \dotsc, \tilde{z}_{\tilde{K}}, z_1, \dotsc, z_K)$ be an interaction function of type $(k_1, \dotsc, k_q)$ and define
\begin{equation}
\label{eq:type_II_expression}
\Lambda_n(\tilde{w}_1, \dotsc, \tilde{w}_{\tilde{K}}, x_1, \dotsc, x_K) =
f([\tilde{x}_{i_1}, \tilde{y}_{j_1}], \dotsc, [\tilde{x}_{i_{\tilde{K}}}, \tilde{y}_{j_{\tilde{K}}}], \lambda_{n, 1}^{\pm}(x_1), \dotsc, \lambda_{n, K}^{\pm}(x_K)).
\end{equation}

Since $\Lambda_n$ is given in terms of a classical interaction function of type $(k_1, \dotsc, k_q)$, solving the PSD problem for type II rate expressions reduces to solving the PSD problem corresponding to this classical interaction function.

Before presenting the general result, let us consider some examples. Consider the pair of type II edges from nodes 1 and 2 terminating on node 3 in Figure~\ref{fig:phosphorylation_edges}(b). The rate expression for  node $3$ is the result of an \textsf{AND} operation on the incoming values of nodes $1$ and $2$, that is, the incoming value to node $3$ is high if and only if the incoming values from nodes $1$ and $2$ are both high. More precisely
\begin{equation}
\label{eq:and}
\Lambda_3(x_1, x_2) = [\lambda_{3,1}^+(x_1),\lambda_{3,2}^+(x_2)] =
\begin{cases}
\ell_{3,(1,2)} & \text{if}~ x_1 < \theta_{3,1} ~\text{or}~ x_2 < \theta_{3,2} \\
\ell_{3,(1,2)} + \delta_{3,(1,2)} & \text{otherwise}
\end{cases}
\end{equation}
where $\ell_{3,(1,2)}$ and $\delta_{3,(1,2)}$ are the parameters corresponding to the pair $(x_1, x_2)$.
Note that while the network in Figure~\ref{fig:phosphorylation_edges}(b) is asymmetric, the net effect on the rate of change of node 2 is symmetric, as expressed in Figure~\ref{fig:phosphorylation_edges}(c).

The factor graph of a node having one or more pairs of type II edges as input is computed from a partial order on the input polynomials computed directly from the original DSGRN total orders. Consider for example, the network in Figure~\ref{fig:phosphorylation_edges}(d). The inputs to node $4$ are nodes $1$, $2$, and $3$ the rate expression for node $4$ is given by
\begin{align*}
& \Lambda_4(x_1, x_2, x_3) =
[\lambda_{4,1}^+(x_1),\lambda_{4,2}^+(x_2)] \lambda_{4,3}^-(x_3) = \\
& \left(
\begin{cases}
\ell_{4,(1,2)} & \text{if}~ x_1 < \theta_{4,1} ~\text{or}~ x_2 < \theta_{4,2} \\
\ell_{4,(1,2)} + \delta_{4,(1,2)} & \text{otherwise}
\end{cases}
\right)
\left(
\begin{cases}
\ell_{4,3} + \delta_{4,3} & \text{if}~ x_3 < \theta_{4,3} \\
\ell_{4,3} & \text{if}~ x_3 > \theta_{4,3}
\end{cases}
\right)
\end{align*}
which takes the following values
\[
\setof*{
\ell_{4,(1,2)} \ell_{4,3}, \, (\ell_{4,(1,2)} + \delta_{4,(1,2)}) \ell_{4,3}, \,
\ell_{4,(1,2)} (\ell_{4,3} + \delta_{4,3}), \,
(\ell_{4,(1,2)} + \delta_{4,(1,2)}) (\ell_{4,3} + \delta_{4,3})
}.
\]
Notice that these are the values of an interaction function of type $(1, 1)$ for the original DSGRN for which the PSD problem have been solved.
In particular, define the polynomials 
\begin{align*}
p_0 := & \ell_{4,(1,2)} \ell_{4,3} &
p_1 := & (\ell_{4,(1,2)} + \delta_{4,(1,2)}) \ell_{4,3} \\
p_2 := & \ell_{4,(1,2)} (\ell_{4,3} + \delta_{4,3}) &
p_3 := & (\ell_{4,(1,2)} + \delta_{4,(1,2)}) (\ell_{4,3} + \delta_{4,3})    
\end{align*}
then for an arbitrary parameter $\mu = (\ell_{4, (1,2)}, \delta_{4, (1,2)}, \ell_{4,3}, \delta_{4,3}) \in (0, \infty)^4$ one of the following orders must be satisfied
\begin{equation}
\label{eq:linear_orders_p}
p_0(\mu) < p_1(\mu) < p_2(\mu) < p_3(\mu) \quad \text{or} \quad p_0(\mu) < p_2(\mu) < p_1(\mu) < p_3(\mu).
\end{equation}
Hence, the admissible linear orders are given by 
\[
(0, 1, 2, 3) \quad \text{and} \quad (0,2,1,3)
\]
which is identical to the solution of the classical PSD problem for interaction type $(1,1)$.

Even though the choice of identical parameters in $\lambda_{4,1}^+$ and $\lambda_{4,2}^+$ reduces the number of values of $\Lambda_4(x_1, x_2, x_4)$ from 8 to 4, in order to compute the state transition graph (STG) we need to be able to determine in which domains these values are attained. 
In other words, in order to construct the STG we need to be able to determine the sign of \eqref{eq:type_II_rate_expression} for the parameter combinations corresponding to all phase space domains. 
Therefore we need keep track of all combinations of the values of $\lambda_{4,1}^+(x_1)$, $\lambda_{4,2}^+(x_2)$, and $\lambda_{4,3}^-(x_3)$ and for each determine the value of $\Lambda_4(x_1, x_2, x_3)$. 
We refer to these combinations as the \emph{input polynomials} corresponding to $\Lambda_4(x_1, x_2, x_3)$. 
Since $[\lambda_{4,1}^+(x_1), \lambda_{4,2}^+(x_2)] = \min \{ \lambda_{4,1}^+(x_1), \lambda_{4,2}^+(x_2) \}$, the input polynomials corresponding to $\Lambda_4(x_1, x_2, x_3)$ are

\begin{align*}
q_0 & := [\ell_{4,(1,2)}, \ell_{4,(1,2)}] \ell_{4,3} = \ell_{4,(1,2)} \ell_{4,3} \\
q_1 & := [\ell_{4,(1,2)} + \delta_{4,(1,2)}, \ell_{4,(1,2)}] \ell_{4,3} = \ell_{4,(1,2)} \ell_{4,3} \\
q_2 & := [\ell_{4,(1,2)}, \ell_{4,(1,2)} + \delta_{4,(1,2)}] \ell_{4,3} = \ell_{4,(1,2)} \ell_{4,3} \\
q_3 & := [\ell_{4,(1,2)} + \delta_{4,(1,2)}, \ell_{4,(1,2)} + \delta_{4,(1,2)}] \ell_{4,3} = (\ell_{4,(1,2)} + \delta_{4,(1,2)}) \ell_{4,3} \\
q_4 & := [\ell_{4,(1,2)}, \ell_{4,(1,2)}] (\ell_{4,3} + \delta_{4,3}) = \ell_{4,(1,2)} (\ell_{4,3} + \delta_{4,3}) \\
q_5 & := [\ell_{4,(1,2)} + \delta_{4,(1,2)}, \ell_{4,(1,2)}] (\ell_{4,3} + \delta_{4,3}) = \ell_{4,(1,2)} (\ell_{4,3} + \delta_{4,3}) \\
q_6 & := [\ell_{4,(1,2)}, \ell_{4,(1,2)} + \delta_{4,(1,2)}] (\ell_{4,3} + \delta_{4,3}) = \ell_{4,(1,2)} (\ell_{4,3} + \delta_{4,3}) \\
q_7 & := [\ell_{4,(1,2)} + \delta_{4,(1,2)}, \ell_{4,(1,2)} + \delta_{4,(1,2)}] (\ell_{4,3} + \delta_{4,3}) = (\ell_{4,(1,2)} + \delta_{4,(1,2)}) (\ell_{4,3} + \delta_{4,3}).
\end{align*}

Since these polynomials only attain the 4 values $p_0,p_1,p_2, p_3$, from the linear orders \eqref{eq:linear_orders_p} we get the following list of partial orders
\[
\{q_0, q_1, q_2\} < q_3 < \{q_4, q_5, q_6\} < q_7 \quad \text{and} \quad
\{q_0, q_1, q_2\} < \{q_4, q_5, q_6\} < q_3 < q_7.
\]
From these partial orders we can compute the factor graph of node $4$ as before by considering all possible values of the thresholds $\setof{\gamma_4 \theta_{j, 4}}$ relative to each of the these partial orders, with the added restriction that valid configurations cannot have thresholds  between $q_0, q_1, q_2$ or between $q_4, q_5, q_6$. Observe that the parameter decomposition is symmetric with respect to the input edges of the type II pair of edges.

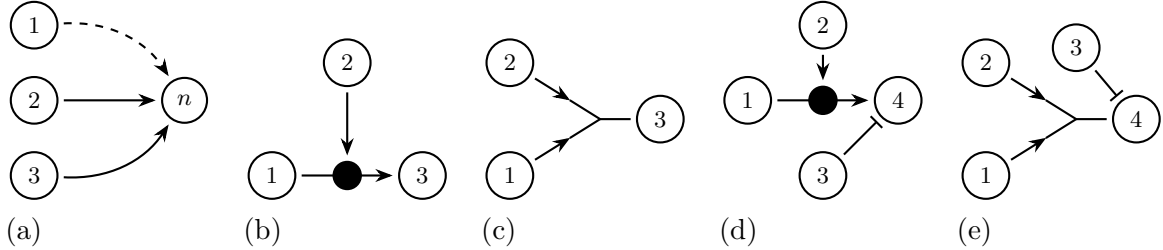
\begin{figure}[!htbp]
\begin{picture}(400,90)
\put(2,0){(a)}
\put(0,15){
\begin{tikzpicture}[thick, main node/.style={circle, fill=white, draw, font=\footnotesize}, scale=1.0]

\node[main node] (1) at (0,0) {$3$};
\node[main node] (2) at (0,1) {$2$};
\node[main node] (3) at (0,2) {$1$};
\node[main node] (4) at (2,1) {$n$};

\draw[-{Stealth}, shorten <= 2pt, shorten >= 2pt] (1) to [bend right] (4);
\draw[-{Stealth}, shorten <= 2pt, shorten >= 2pt] (2) -- (4);
\draw[-{Stealth}, dashed, shorten <= 2pt, shorten >= 2pt] (3) to [bend left] (4);
\end{tikzpicture}
}
\put(92,0) {(b)}
\put(90,15){
\begin{tikzpicture}[thick, main node/.style={circle, fill=white, draw, font=\footnotesize}, mid+ node/.style={circle, fill=black, draw}, scale=1.0]

\node[main node] (1) at (0,0) {1};
\node[main node] (2) at (1,1.5) {2};
\node[main node] (3) at (2,0) {3};
\node[mid+ node] (4) at (1,0) {};

\draw[-{Stealth}, shorten <= 2pt, shorten >= 2pt] (1) -- (3);
\draw[-{Stealth}, shorten <= 2pt, shorten >= 2pt] (2) -- (4);
\end{tikzpicture}
}
\put(182,0){(c)}
\put(180,15){
\begin{tikzpicture}[thick, main node/.style={circle, fill=white, draw, font=\footnotesize}, fake node/.style={circle, inner sep=0pt}, scale=1.0]

\node[main node] (1) at (0,0) {1};
\node[main node] (2) at (0,1.5) {2};
\node[main node] (3) at (2,0.75) {3};
\node[fake node] (4) at (0.8,0.5) {};
\node[fake node] (5) at (0.8,1) {};
\node[fake node] (6) at (1.2,0.75) {};

\draw[-{Stealth}, shorten <= 2pt, shorten >= -1pt] (1) -- (4);
\draw[-{Stealth}, shorten <= 2pt, shorten >= -1pt] (2) -- (5);
\draw[-, shorten <= -1pt, shorten >= -1pt] (4) -- (6);
\draw[-, shorten <= -1pt, shorten >= -1pt] (5) -- (6);
\draw[-, shorten <= -1pt, shorten >= 2pt] (6) -- (3);
\end{tikzpicture}
}
\put(272,0){(d)}
\put(270,15){
\begin{tikzpicture}[thick, main node/.style={circle, fill=white, draw, font=\footnotesize}, mid+ node/.style={circle, fill=black, draw}, scale=1.0]

\node[main node] (1) at (0,1) {1};
\node[main node] (2) at (1,2) {2};
\node[main node] (3) at (1,0) {3};
\node[main node] (4) at (2,1) {4};
\node[mid+ node] (5) at (1,1) {};

\draw[-{Stealth}, shorten <= 2pt, shorten >= 2pt] (1) -- (4);
\draw[-|, shorten <= 2pt, shorten >= 2pt] (3) -- (4);
\draw[-{Stealth}, shorten <= 2pt, shorten >= 2pt] (2) -- (5);
\end{tikzpicture}
}
\put(362,0){(e)}
\put(360,15){
\begin{tikzpicture}[thick, main node/.style={circle, fill=white, draw, font=\footnotesize}, fake node/.style={circle, inner sep=0pt}, scale=1.0]

\node[main node] (1) at (0,0) {1};
\node[main node] (2) at (0,1.5) {2};
\node[main node] (3) at (1.2,1.7) {3};
\node[main node] (4) at (2,0.75) {4};
\node[fake node] (5) at (0.8,0.5) {};
\node[fake node] (6) at (0.8,1) {};
\node[fake node] (7) at (1.2,0.75) {};

\draw[-{Stealth}, shorten <= 2pt, shorten >= -1pt] (1) -- (5);
\draw[-{Stealth}, shorten <= 2pt, shorten >= -1pt] (2) -- (6);
\draw[-, shorten <= -1pt, shorten >= -1pt] (5) -- (7);
\draw[-, shorten <= -1pt, shorten >= -1pt] (6) -- (7);
\draw[-, shorten <= -1pt, shorten >= 2pt] (7) -- (4);
\draw[-|, shorten <= 2pt, shorten >= 2pt] (3) -- (4);
\end{tikzpicture}
}
\end{picture}
\caption{(a) Example of type I regulation; (b) Type II edges terminating on node $3$; (c) Implementation of interaction between type II edges starting from nodes $1$ and $2$ to node $3$; (d) More complex interaction between type II and type 0 inputs to node $4$; (e) Implementation of (d).}
\label{fig:phosphorylation_edges}
\end{figure}

We now present the general result for type II edges. The proof is a simple generalization of the ideas presented in the above example and it is left to the reader.

\begin{thm}
\label{thm:type_II_PSD}
Suppose node $n$ is governed by a type II rate expression defined by an interaction function $f$ of order $K$ and type $(k_1, \dotsc, k_q)$ as defined in equation \eqref{eq:type_II_expression}. Then, the PSD problem associated to $\Lambda_n$ is equivalent to the classical PSD problem of type $(k_1, \dotsc, k_q)$ associated to the interaction function $f$.
\end{thm}

Observe  that the parameter decomposition and the DSGRN computations are  symmetric with respect to the input edges to a type II pair of edges. For this reason internally DSGRN represents pairs of type II edges as in Figure~\ref{fig:phosphorylation_edges}(c) and (e).

For type II edges we refer to the polynomial expression of $f$ as the \textit{interaction type of a type II rate expression}. In addition, we enclose the pairs of type II edges in between square brackets $[ \, \cdot \, ]$ to denote the pair of edges used to define the rate expression \eqref{eq:type_II_expression}.
Notice that the notation $[x, y]$ is symmetric for a pair of type II edges $x$ and $y$.
As an example, the interaction type of the incoming edges to node 4 in Figure~\ref{fig:phosphorylation_edges}(e) is denoted by
\[
[x_1, x_2] x_3.
\]

Observe that the linear orders for this structure are the same as the linear orders \eqref{eq:linear_orders_p} for the regular DSGRN interaction type
\[
z_1 z_2.
\]

By \eqref{eq:fulldot} and \eqref{eq:emptydot}, the only effect of an  open dot interaction, compared to full dot interaction, between a pair of type II edges is to change the sign of the second edge, that is, it changes the second edge from $i \to j$ to $i \dashv j$ or from $i \dashv j$ to $i \to j$. In DSGRN this change of sign is denoted by
\[
[x_1, \sim \! x_2].
\]

We can compute  a complete list of interaction types involving type II edges supported by DSGRN by the following procedure: Take any interaction type from \cite[Table 1]{kepley:mischaikow:zhang} or from Table~\ref{table:ubiquitination_logic} (which in turn is derived from \cite[Table 1]{kepley:mischaikow:zhang}) and replace one or more nodes by a pair of nodes representing type II edges (see also Section~\ref{sec:types_I_II}). The size of the factor graph thus generated is the same as the size of the factor graphs of the corresponding interaction type in the original DSGRN.

For example, the original DSGRN interaction type $x y$ can be used to generate the type I interaction type
\[
\langle x \rangle + y
\]
and from these we can derive the following interaction types involving type II edges
\[
[x_1, x_2] y, \quad [x_1, x_2] [y_1, y_2], \quad
\langle [x_1, x_2] \rangle + y, \quad \langle x \rangle + [y_1, y_2], \quad \langle [x_1, x_2] \rangle + [y_1, y_2],
\]
all of which have factor graphs of same size as $x y$.

Table~\ref{table:phosphorylation_logic} presents a complete list of the interaction types derived from $x(y+z)$ as well as some additional examples of interaction types involving type I and type II edges supported by DSGRN.

\subsection{Combination of type I and type II  edges}
\label{sec:types_I_II}

As indicated in the introduction, type I and II  edges can be combined. In particular, a pair of type II edges can affect the decay of a node. As a result in a rate expression of type I we can replace one or more of the type I edges by pairs of type II edges. To solve a PSD problem and hence find the linear orders for this type of expression, we first get the linear orders for the corresponding type I expression and then apply the procedure described in Section~\ref{sec:typeIIedges} that computes  the partial orders for the type II edges.
Consider, for example, the following interaction type involving type I expression
\[
\langle u + z \rangle + v w.
\]

We can apply the procedure described in Section~\ref{sec:typeIIedges} to the linear orders of this interaction type to get the partial orders of the interaction type where we replace the edge $u$ by the type II pair $[x, y]$ to get the partial orders for the interaction type
\[
\langle [x, y] + z \rangle + v w.
\]

In a similar way we  can obtain the linear orders for interaction types involving all three types of edges such as
\[
\langle [x, y] + z \rangle + u [v, w].
\]

\begin{table}[!htbp]
\centering
\renewcommand{\arraystretch}{1.2}
\begin{tabular}{@{}ll@{}}
\toprule
\midrule
PTM interaction type & DSGRN interaction type \\
\midrule
$[x_1, x_2](y+z)$ & $x (y + z)$ \\
$x([y_1, y_2]+z)$ & $x (y + z)$ \\
$[x_1, x_2]([y_1, y_2]+z)$ & $x (y + z)$ \\
$x([y_1, y_2]+[z_1, z_2])$ & $x (y + z)$ \\
$[x_1, x_2] ([y_1, y_2]+[z_1, z_2])$ & $x (y + z)$ \\
$\langle x \rangle + ([y_1, y_2]+z)$ & $x (y + z)$ \\
$\langle x \rangle + ([y_1, y_2]+[z_1, z_2])$ & $x (y + z)$ \\
$\langle [x_1, x_2] \rangle + (y+z)$ & $x (y + z)$ \\
$\langle [x_1, x_2] \rangle + ([y_1, y_2]+z)$ & $x (y + z)$ \\
$\langle [x_1, x_2] \rangle + ([y_1, y_2]+[z_1, z_2])$ & $x (y + z)$ \\
$\langle y+z \rangle + [x_1, x_2]$ & $x (y + z)$ \\
$\langle [y_1, y_2]+z \rangle + x$ & $x (y + z)$ \\
$\langle [y_1, y_2]+z \rangle + [x_1, x_2]$ & $x (y + z)$ \\
$\langle [y_1, y_2]+[z_1, z_2] \rangle + x$ & $x (y + z)$ \\
$\langle [y_1, y_2]+[z_1, z_2] \rangle + [x_1, x_2]$ & $x (y + z)$ \\
\midrule
$[x, y] + z$ & $x + z$ \\
$[x, y]  z$ & $x  z$ \\
$\langle x \rangle + [y, z]$ & $x y$ \\
$\langle x \rangle + [y, z] w$ & $xyw$ \\
$\langle x \rangle + [y, z] u w$ & $xyuw$ \\
$\langle x \rangle + [y, z] [u, v] w$ & $xyuw$ \\
$x y([z, w]+u)$ & $xy(z+u)$ \\
$\langle xy \rangle + [z,w] + u$ & $xy(z+u)$ \\
$\langle xy \rangle + [z,w] + [u, v]$ & $xy(z+u)$ \\
\bottomrule
\end{tabular}
\caption{Complete list of the interaction types derived from $x(y+z)$ (top) and some additional examples of interaction types with type I and type II edges supported by DSGRN. The sizes of the factor graphs for each of the interaction types are the same as the size of the corresponding original DSGRN interaction type listed.}
\label{table:phosphorylation_logic}
\end{table}

\section{Computational Examples}
\label{sec:examples}

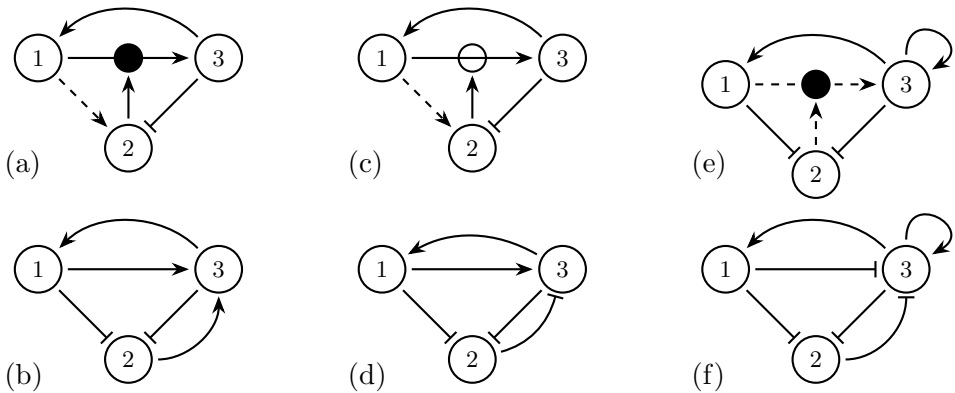
\begin{figure}
\centering
\begin{picture}(400,160)
\put(0,90){(a)}
\put(0,90){
\begin{tikzpicture}[thick, main node/.style={circle, fill=white, draw, font=\footnotesize}, mid+ node/.style={circle, fill=black, draw}, fake node/.style={circle}, scale=1.2]

\node[main node] (1) at (0,1) {1};
\node[main node] (3) at (2,1) {3};
\node[main node] (2) at (1,0) {2};
\node[mid+ node] (5) at (1,1) {};

\draw[-{Stealth}, shorten <= 2pt, shorten >= 2pt] (1) -- (3);
\draw[-{Stealth}, shorten <= 2pt, shorten >= 1pt] (2) -- (5);
\draw[-{Stealth}, dashed, shorten <= 2pt, shorten >= 2pt] (1) -- (2);
\draw[-|, shorten <= 2pt, shorten >= 2pt] (3) -- (2);
\draw[-{Stealth}, shorten <= 2pt, shorten >= 2pt] (3) to [out=135, in=45, looseness=1] (1);
\end{tikzpicture}
}
\put(0,10){(b)}
\put(0,10){
\begin{tikzpicture}[thick, main node/.style={circle, fill=white, draw, font=\footnotesize}, mid+ node/.style={circle, fill=black, draw}, fake node/.style={circle}, scale=1.2]

\node[main node] (1) at (0,1) {1};
\node[main node] (3) at (2,1) {3};
\node[main node] (2) at (1,0) {2};

\draw[-{Stealth}, shorten <= 2pt, shorten >= 2pt] (1) -- (3);
\draw[-|, shorten <= 2pt, shorten >= 2pt] (1) -- (2);
\draw[-|, shorten <= 2pt, shorten >= 2pt] (3) -- (2);
\draw[-{Stealth}, shorten <= 2pt, shorten >= 1pt] (2) to [out=0, in=-90, looseness=1] (3);
\draw[-{Stealth}, shorten <= 2pt, shorten >= 2pt] (3) to [out=135, in=45, looseness=1] (1);
\end{tikzpicture}
}
\put(130,90){(c)}
\put(130,90){
\begin{tikzpicture}[thick, main node/.style={circle, fill=white, draw, font=\footnotesize}, mid+ node/.style={circle, fill=black, draw}, fake node/.style={circle}, scale=1.2]

\node[main node] (1) at (0,1) {1};
\node[main node] (3) at (2,1) {3};
\node[main node] (2) at (1,0) {2};
\node[main node] (5) at (1,1) {};

\draw[-{Stealth}, shorten <= 2pt, shorten >= 2pt] (1) -- (3);
\draw[-{Stealth}, shorten <= 2pt, shorten >= 1pt] (2) -- (5);
\draw[-{Stealth}, dashed, shorten <= 2pt, shorten >= 2pt] (1) -- (2);
\draw[-|, shorten <= 2pt, shorten >= 2pt] (3) -- (2);
\draw[-{Stealth}, shorten <= 2pt, shorten >= 2pt] (3) to [out=135, in=45, looseness=1] (1);
\end{tikzpicture}
}
\put(130,10){(d)}
\put(130,10){
\begin{tikzpicture}[thick, main node/.style={circle, fill=white, draw, font=\footnotesize}, mid+ node/.style={circle, fill=black, draw}, fake node/.style={circle}, scale=1.2]

\node[main node] (1) at (0,1) {1};
\node[main node] (3) at (2,1) {3};
\node[main node] (2) at (1,0) {2};

\draw[-{Stealth}, shorten <= 2pt, shorten >= 2pt] (1) -- (3);
\draw[-|, shorten <= 2pt, shorten >= 2pt] (1) -- (2);
\draw[-|, shorten <= 2pt, shorten >= 2pt] (3) -- (2);
\draw[-|, shorten <= 2pt, shorten >= 1pt] (2) to [bend right] (3);
\draw[-{Stealth}, shorten <= 2pt, shorten >= 2pt] (3) to [bend right] (1);
\end{tikzpicture}
}
\put(260,90){(e)}
\put(260,80){
\begin{tikzpicture}[thick, main node/.style={circle, fill=white, draw, font=\footnotesize}, mid+ node/.style={circle, fill=black, draw}, fake node/.style={circle}, scale=1.2]

\node[main node] (1) at (0,1) {1};
\node[main node] (3) at (2,1) {3};
\node[main node] (2) at (1,0) {2};
\node[mid+ node] (5) at (1,1) {};

\draw[-{Stealth}, dashed, shorten <= 2pt, shorten >= 2pt] (1) -- (3);
\draw[-{Stealth}, dashed, shorten <= 2pt, shorten >= 1pt] (2) -- (5);
\draw[-|, shorten <= 2pt, shorten >= 2pt] (1) -- (2);
\draw[-|, shorten <= 2pt, shorten >= 2pt] (3) -- (2);
\draw[-{Stealth}, shorten <= 2pt, shorten >= 2pt] (3) to [out=135, in=45, looseness=1] (1);
\draw[-{Stealth}, shorten <= 2pt, shorten >= 2pt] (3) to [out=90, in=30, looseness=6] (3);
\end{tikzpicture}
}
\put(260,10){(f)}
\put(260,10){
\begin{tikzpicture}[thick, main node/.style={circle, fill=white, draw, font=\footnotesize}, mid+ node/.style={circle, fill=black, draw}, fake node/.style={circle}, scale=1.2]

\node[main node] (1) at (0,1) {1};
\node[main node] (3) at (2,1) {3};
\node[main node] (2) at (1,0) {2};

\draw[-|, shorten <= 2pt, shorten >= 2pt] (1) -- (3);
\draw[-|, shorten <= 2pt, shorten >= 2pt] (1) -- (2);
\draw[-|, shorten <= 2pt, shorten >= 2pt] (3) -- (2);
\draw[-|, shorten <= 2pt, shorten >= 1pt] (2) to [out=0, in=-90, looseness=1] (3);
\draw[-{Stealth}, shorten <= 2pt, shorten >= 2pt] (3) to [out=135, in=45, looseness=1] (1);
\draw[-{Stealth}, shorten <= 2pt, shorten >= 2pt] (3) to [out=90, in=30, looseness=6] (3);
\end{tikzpicture}
}
\end{picture}
\caption{ In each column the top network has type I and type II edges and at the bottom is the same network interpreted with type $0$ edges preserving the sign of the interaction. A dashed edge indicates that its end effect is on the decay rate of the receiving network node.}
\label{fig:examples}
\end{figure}

In this section we present three examples of small networks with type I and type II edges, see top row of Figure~\ref{fig:examples}. The bottom row of Figure~\ref{fig:examples} are the corresponding networks where type 1 and 2 edges were replaced by type 0 edges.

In Figure~\ref{fig:examples}(a) the $3 \to 1$ and $3 \dashv 2$ are type 0 edges. The edge $1 \to 2$ is type I edge where node $1$ up-regulates the decay of $2$. Node $2$ changes  activity of $1$ via a type II edge. The fact that the  edge $1 \to 3$ is up-regulating means that the active version of $1$ up-regulates $3$; the fact that the node on this edge is filled means that the effect of $2$ is to change $1$ from its inactive form to an active form. 

In Figure~\ref{fig:examples}(b) is the network from (a) with only type $0$ edges that captures the same overall effect between the nodes. In particular, since type 1 edge $1\to 2$  in (a) up-regulates decay, it is replaced by a negative type $0$ edge. The type II edges from $1$ and $2$ that affect (3) are replaced by a pair of positive type $0$ edges.

In Figure~\ref{fig:examples}(c) the only change in comparison to (a) is that the node in the $1\to 3$ edge is empty, which means that $2$ facilitates transition from active to inactive form of $1$. As a result, the overall effect from $2$ to $3$ is negative and therefore in Figure~\ref{fig:examples}(d) this influence is represented by a negative type $0$ edge.

The last pair of networks is in Figures~\ref{fig:examples}(e) and (f). 
In (e) node $2$ up-regulates the transition from the inactive to the active form of $1$, which promotes decay of $3$. Therefore the corresponding network in Figure~\ref{fig:examples}(f) has two negative edges $2\dashv 3$ and $1 \dashv 3$.

As described above for each network we compute the decomposition of parameter space into domains (i.e  parameter nodes) and for each of them we compute the state transition graph and the Morse graph. We characterize Morse nodes by their dynamic phenotypes: 
\begin{itemize}
\item When a Morse node consists of a single domain $\kappa$ with  a self-edge, we designate it FP (for ``Fixed Point'');
\item When a Morse node consists of a set of domains with a path such that along that path at least one threshold $\theta_{*,i}$ is crossed for every regulatory network vertex $i$, we designate the Morse node FC (for ``Full Cycle'');
\item If the only paths within a Morse node are such that a proper subset of variables cross thresholds along a path, we designate such Morse node PC (for ``Partial Cycle'').
\end{itemize}
We are most interested in describing attractors in dynamics and therefore we concentrate on Morse nodes that are leaves in the Morse graph; we call such Morse nodes stable. 
 Table~\ref{table:dsgrn_computations} presents basic statistics on the number of stable Morse nodes  (FC, PC, and FP) for the networks in Figure~\ref{fig:examples}.

\begin{table}[!htbp]
\centering
\renewcommand{\arraystretch}{1.2}
\begin{tabular}{@{}lllllll@{}}
\toprule
\midrule
Network & Fig.~\ref{fig:examples}(a) & Fig.~\ref{fig:examples}(b) & Fig.~\ref{fig:examples}(c) & Fig.~\ref{fig:examples}(d) & Fig.~\ref{fig:examples}(e) & Fig.~\ref{fig:examples}(f) \\
\midrule
Number of Parameters & $864$ & $2880$ & $864$ & $2880$ & $21600$ & $305424$ \\
\midrule
Stable FC & $1.39\%$ & $4.58\%$ & $0\%$ & $0\%$ & $0\%$ & $1.49\%$ \\
Stable PC & $4.17\%$ & $9.17\%$ & $0\%$ & $0\%$ & $3.83\%$ & $7.85\%$ \\
\midrule
0 stable FP & $5.55\%$ & $13.19\%$ & $0\%$ & $0\%$ & $3.07\%$ & $7.62\%$ \\
1 stable FP & $90.74\%$ & $78.89\%$ & $87.04\%$ & $74.72\%$ & $61.59\%$ & $56.41\%$ \\
2 stable FP & $3.70\%$ & $7.92\%$ & $12.96\%$ & $23.61\%$ & $33.0\%$ & $31.16\%$ \\
3 stable FP & $0\%$ & $0\%$ & $0\%$ & $1.67\%$ & $2.33\%$ & $4.81\%$ \\
\bottomrule
\end{tabular}
\caption{Computations for the networks in Figure~\ref{fig:examples}. The second row from the top shows the number of parameter nodes in the parameter graph of the corresponding network. The remaining rows show the percentage of parameter nodes with the type of dynamics indicated in the first column.}
\label{table:dsgrn_computations}
\end{table}

Comparing paired networks, (a)-(b), (c)-(d), and (e)-(f), we observe that the size of the parameter graphs are different. In addition, the number of parameter nodes with stable FC, stable PC, zero stable FP, one stable FP, bi-stability and tri-stability between FPs, are all different. In addition, if these counts are expressed  as a percentage of all parameter nodes, the percentages are also all different. 

We conclude that the new modality of DSGRN of modeling type I and type II edges in addition to type $0$ edges can lead to different results in characterization of network dynamics. When the type I and type II edges
describe more closely the underlying biological system, these results may offer a more detailed understanding of the network function.

\section{Singular Limits of ODE Models}
\label{sec:ODE}

Up to this point in the paper we have consciously avoided writing down explicit ODE models.
The motivation for this is twofold: (i) the focus of the paper is on the challenge {\bf C2}, developing a computational framework in which to identify the global dynamics of complex regulatory networks, and (ii) to emphasize that the results are independent of specific models or applications.
However, as indicated in the introduction in the context of applications we must also address the challenge {\bf C1}.
Our primary motivation for developing the DSGRN software has come from systems biology and therefore we return to this subject to identify how a user may make use of the extended DSGRN tools to analyze networks involving transcriptional and post-transcriptional regulation.

Recall that a common ODE model describing the up regulation  of production of the protein of gene $n$ by the protein of  gene $m$ has form
\begin{equation}
\label{eq:hilln}
    \dot{x}_n = -\gamma_n x_n + \ell_{n,m} + \delta_{n,m} \frac{x_{m}^h}{\theta_{n,m}^h + x_{m}^h}
\end{equation}
where the nonlinearity is the classical Hill function.
Even this minimal model has a multitude of parameters $\gamma$, $\theta$, $\ell$, $\delta$ and $h$. 
In the search for qualitative insight, it makes sense to sacrifice functional form for ease of analysis. 
One direction is to consider the limit $h\to\infty$ that results in an equation of the form
\begin{equation}
\label{eq:switching}
    \dot{x}_n = -\gamma_n x_n + \begin{cases}
    \ell_{n,m} & \text{if $x_m < \theta_{n,m}$} \\
    \ell_{n,m} +\delta_{n,m} & \text{if $x_m > \theta_{n,m}$.} 
    \end{cases} 
     = -\gamma_n x_n +\lambda^+(x_m; \ell_{n,m},\delta_{n,m}, \theta_{n,m}).
\end{equation}
where notation $\lambda^+$ extends definition  (\ref{eq:lambdaplusminus})
by including explicit parameter values in the argument.
Similarly, a limit $h\to\infty$  of the equation 
\begin{equation}
\label{eq:hilln-neg}
    \dot{x}_n = -\gamma_n x_n + \ell_{n,m} + \delta_{n,m} \frac{\theta_{n,m}^h}{\theta_{n,m}^h + x_{m}^h}
\end{equation}
results in an equation 
\begin{equation}
\label{eq:switching2}
    \dot{x}_n = -\gamma_n x_n + \begin{cases}
    \ell_{n,m} & \text{if $x_m > \theta_{n,m}$} \\
    \ell_{n,m} +\delta_{n,m} & \text{if $x_m < \theta_{n,m}$.} 
    \end{cases} 
     = -\gamma_n x_n +\lambda^-(x_m; \ell_{n,m},\delta_{n,m}, \theta_{n,m}).
\end{equation}

Models of regulatory networks using equations of the form of \eqref{eq:switching}-\eqref{eq:switching2} are typically called switching systems and have been analyzed and applied~\cite{Thomas1991,edwards00,deJong2002,velfingstad07} since their introduction by Glass and Kaufmann \cite{GL_original_paper,Glass1972}.
The development of the current version of the DSGRN software was motivated by switching systems.
However, there is a fundamental difference in that we are not interested in and do not solve for trajectories of switching system ODEs.
Instead, as indicted in Section~\ref{sec:intro} we retreat to a combinatorial model that only makes use of the sign of the right hand side, e.g.\ \eqref{eq:toggleSwitch}.
The more general $\Lambda_n$ described in Section~\ref{sec:dsgrn} can all be obtained in a similar manner, by starting with a product of sums of Hill functions, each of which expresses an edge interaction, 
and allowing  the Hill exponents $h$ to become 
arbitrarily large.

As is discussed in \cite{gameiro:gedeon:kepley:mischaikow} and  \cite{diegmiller}
the dynamics captured by the DSGRN computations is highly suggestive of the dynamics exhibited by the ODEs even for moderate levels of exponent $h$ in the Hill functions.

To provide motivation for the expressions for the extended version of DSGRN developed in this paper we turn to the theory of multi-site protein control.
For simplicity of notation we assume that the protein being controlled is produced by Node 1 and that the associated protein, Protein 1, is controlled by Proteins 2 and 3.
\begin{description}
    \item[P1] Protein 1 has $M$ sites and the activity level  of Protein 1 is governed by how many sites are filled. In particular, Protein 1 is $\mathsf{off}$ if less than $K$ of the sites are occupied and Protein 1 is $\mathsf{on}$ if $K$ or more sites are occupied.
    \item[P2] Protein 2 and Protein 3 act as enzymes, with Protein 2 filling sites and Protein 3 emptying sites.
\end{description}
A guiding example for this arrangement may be multi-site phosphorylation  by  a kinase (Protein 2) and de-phosphorylation by a phosphatase (Protein 3). 
We denote  the concentration of Proteins 1 - 3 by $x_1$, $x_2$, and $x_3$, respectively.
To be more precise we let $x_1^{(m)}$, $m=0,\ldots, M$, denote the concentration of Protein 1, that has $m$ sites filled.
Motivated by assumption {\bf P1} we set
\[
x_1^{\mathsf{off}} = \sum_{m=0}^{K-1}x_1^{(m)}\quad\text{and}\quad
x_1^{\mathsf{on}} = \sum_{m=K}^{M}x_1^{(m)}
\]
We model assumption {\bf P2} by the following system of chemical reactions
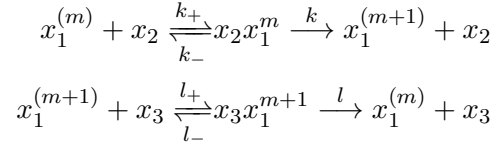
\begin{align*}
    x_1^{(m)} + x_2 \xrightleftharpoons[k_-]{k_+} & x_2x_1^m  \mapright{k} x_1^{(m+1)} + x_2 \\
    x_1^{(m+1)} + x_3 \xrightleftharpoons[l_-]{l_+} & x_3x_1^{m+1}  \mapright{l} x_1^{(m)} + x_3
\end{align*}
and set
\[
\alpha =   \frac{kk_+(l+l_-)}{ll_+(k+k_-)}. 
\]
Following the  analysis in  \cite{wang:nie:enciso}\footnote{This paper also provides biological relevance for this modeling approach.}
we obtain that for $K=2M-1$, the concentration of the active version $ x_1^{\mathsf{on}}$ of $x_1$ is given by
\begin{equation}
    \label{eq:genericRatio}
    x_1^{\mathsf{on}} = \frac{(\alpha\frac{x_2}{x_3})^K}{1 + (\alpha\frac{ x_2}{x_3})^K} \, x_1.
\end{equation}
Finally, taking the limit as $M\to \infty$ while keeping $M=2K-1$ allows us to write $x_1^{\mathsf{on}}$ as a function of the variable $\alpha \frac{x_2}{x_3}$
\begin{equation}
\label{eq:genericstep}
x_1^{\mathsf{on}}\left(\alpha \frac{x_2}{x_3} \right)  =
\sigma^+ \left(\alpha \frac{x_2}{x_3};1 \right) x_1 = \sigma^+ \left(\frac{x_2}{x_3}; \frac{1}{\alpha} \right)x_1
\end{equation}
where
\[
\sigma^+(\xi;\zeta):= \begin{cases} 
0 & \text{if $\xi < \zeta$} \\
1 & \text{if $\xi > \zeta$}.
\end{cases}
\]  

\subsection{Type I: Decay control edges}

As indicated in the introduction we  use the analysis from the previous Section to  model the  modulation of the decay rate, see  Figure~\ref{fig:RN+}(b). In biological setting, the decay up-regulation is accomplished by ubiquitination by an enzyme ubiquitin ligase, and decay down-regulation by  de-ubiquitination by ubiquitin-specific protease.

Let $x_1$ denote the total concentration of quantity associated to node 1 that is undergoing up-regulation of its decay rate by $x_3$ and down-regulation of its decay rate by $x_2$.
We restrict our attention to processes where only one of $x_2$ and $x_3$ is actively controlled. 
We first focus on control of the decay up-regulation
and assume 
\begin{description}
\item[A3.1] \textit{ The concentration of  decay down-regulator 
$x_3$ is a constant $y$. }
\end{description}
Under assumptions {\bf A1},  {\bf A2} and {\bf A3.1} we can rewrite \eqref{eq:genericstep} as
\begin{equation}
    \label{eq:ubquitination}
    x_1^{\mathsf{on}}(x_2) = \sigma^+(x_2;\beta_{1,2})x_1
\end{equation}
where $\beta_{1,2} = y/\alpha$.

To incorporate this into DSGRN we assume that the decay rate for $x_1$ 
in the absence of $x_2$ is $\gamma_1$. 
Thus \eqref{eq:switching} becomes 
\begin{align*}
    \dot{x}_1 &= -\gamma_1 x_1 + \delta_{1,2}\sigma^+(x_2;\beta_{1,2})x_1 + \Lambda_1(x) = 
    -\left(
    \begin{cases}
    \gamma_1 & \text{if $x_2 < \beta_{1,2}$} \\
    \gamma_1 +\delta_{1,2} & \text{if $x_2 > \beta_{1,2}$}
    \end{cases}
    \right)x_1 + \Lambda_1(x)\nonumber \\
    &=
    -\lambda^+(x_2; \gamma_1, \delta_{1,2},\beta_{1,2})x_1 + \Lambda_1(x)
\end{align*}

We model decay down-regulation 
in a similar manner. We make an assumption
\begin{description}
\item[A3.2] \textit{The concentration of  decay up-regulator 
$x_2$ is constant $y$.}
\end{description}
Under assumptions {\bf A1},  {\bf A2} and {\bf A3.2} can rewrite \eqref{eq:genericstep} as
\begin{equation}
    \label{eq:deubquitination}
    x_1^{\mathsf{off}}(x_3) = \left(1-\sigma^+(x_3;\beta_{1,3})\right)x_1
\end{equation}
where $\beta_{1,3} = \alpha y$.
 We rewrite the equation (\ref{eq:deubquitination}) one more time as 
\begin{equation}
    \label{eq:deubquitination2}
    x_1^{\mathsf{off}}(x_3) =\sigma^-(x_3;\beta_{1,3})x_1
\end{equation}
where
\[
\sigma^-(\xi;\zeta):= \begin{cases} 
    1 & \text{if $\xi < \zeta$} \\
    0 & \text{if $\xi > \zeta$} .
    \end{cases}
\]    
To incorporate this into DSGRN we assume that the decay rate for Protein 1 in the absence of $x_3$ is $\gamma_1$.
Then \eqref{eq:switching} becomes 
\begin{align*}
    \dot{x}_1 &= -(\gamma_1 + \delta_{1,3}\sigma^-(x_3;\beta_{1,3}))x_1 + \Lambda_1(x) = 
    -\left(
    \begin{cases}
    \gamma_1 & \text{if $x_3 > \beta_{1,3}$} \\
    \gamma_1 +\delta_{1,3} & \text{if $x_3 < \beta_{1,3}$}
    \end{cases}
    \right)x_1 + \Lambda_1(x)\nonumber \\ 
    &=-\lambda^-(x_3;\gamma_1,\delta_{1,3},\beta_{1,3})x_1 + \Lambda_1(x). 
\end{align*}

\subsection{Type II: Activity control edges}

We now turn our attention to modeling modulation of 
activity level, see Figure~\ref{fig:RN+}(c)-(d). In the biological context there are several modifications  that affect activity of a protein.
This  is often achieved by binding an additional molecule or a group to an existing protein and modifying its properties. Examples include  phosphorylation, methylation, glycosylation, lipidation  and  other modifications.

We will first consider interactions shown in Figure~\ref{fig:RN+}(c) where the $x_1^\mathsf{on}$   is the active form of $x_1$. In the same way as we did for the  decay control processes, we assume that only one of the processes  controlling the  activation vs. deactivation is actively controlled.  We start with 
\begin{description}
\item[A3.3] \textit{The concentration $x_3$ remains constant.}
\end{description}
Under assumptions {\bf A1},  {\bf A2} and {\bf A3.3} we can rewrite \eqref{eq:genericstep} as
\begin{equation}
    \label{eq:phosphorylation}
    x_1^{\mathsf{on}}(x_2) = \sigma^+(x_2;\beta_{1,2})x_1
\end{equation}
where $\beta_{1,2} = x_3/\alpha$.
On the other hand, if we assume 
\begin{description}
\item[A3.4] \textit{The concentration $x_2$ remains constant.}
\end{description}
Then, under assumptions {\bf A1},  {\bf A2} and {\bf A3.4} we can rewrite \eqref{eq:genericstep} as
\begin{equation}
    \label{eq:dephosphorylation}
    x_1^{\mathsf{on}}(x_3) = \sigma^+ \left( \frac{1}{x_3};\frac{1}{\beta_{1,3}} \right) x_1
    = \sigma^-(x_3; \beta_{1,3})x_1
\end{equation}
where $\beta_{1,3} = \alpha x_2$.

To incorporate this into DSGRN we assume that the production rates for $x_4$ by $x_1$, when $x_1$ is in the inactive state  $\mathsf{off}$ and active state $\mathsf{on}$ are given by $\ell_{4,1}$ and $\ell_{4,1}+\delta_{4,1}$, respectively.

Then (\ref{eq:switching}) for $n=4$ is
\begin{align} \label{eq:x4equation}
    \dot{x}_4 &= -\gamma_4 x_4 + \Lambda_4(x_1)\nonumber  \\
   &= -\gamma_4 x_4 + \begin{cases} 
\ell_{4,1} & \text{if $x_1 < \theta_{4,1}$}\\
\ell_{4,1} + \delta_{4,1} & \text{if $x_1 > \theta_{4,1}$}
\end{cases}
\end{align}

Under the assumption {\bf A3.3} this becomes 
\begin{align*}
\dot{x}_4 &= -\gamma_4 x_4 + \Lambda_4(\sigma^+(x_2;\beta_{1,2})x_1)\\
&= -\gamma_4 x_4 + 
\begin{cases}
\ell_{4,1} & \text{if $x_2 < \beta_{1,2}$}\\
\ell_{4,1} & \text{if $x_1 < \theta_{4,1}$}\\
\ell_{4,1} + \delta_{4,1} & \text{if $x_1 > \theta_{4,1}$ and $x_2 >\beta_{1,2}$.}
\end{cases}\\
&=-\gamma_4 x_4 + [\lambda^+(x_1),\lambda^+(x_2)],
\end{align*}
where notation $[\cdot,\cdot]$ has been introduced in \eqref{eq:and}.
 
On the other hand, under the assumption {\bf A3.4} we have
\begin{align*}
     \dot{x}_4 &= -\gamma_4 x_4 + \Lambda_4(\sigma^-(x_3;\beta_{1,3})x_1)\\
     &=  -\gamma_4 x_4 + \begin{cases}
    \ell_{4,1} & \text{if $x_3 >\beta_{1,3}$}\\
    \ell_{4,1} & \text{if $x_1 < \theta_{4,1}$}\\
    \ell_{4,1} + \delta_{4,1} & \text{if $x_1 > \theta_{4,1}$ and $x_3 < \beta_{1,3}$.}
    \end{cases}\\
    &=  -\gamma_4 x_4 + [\lambda^+(x_1),\lambda^-(x_3)]
     \end{align*}
     
    For the second type of interactions in Figure~\ref{fig:RN+}(c) the equation (\ref{eq:x4equation}) has reversed inequalities with respect to threshold $\theta_{4,1}$
\begin{align*}
    \dot{x}_4 &=  -\gamma_4 x_4 + \begin{cases} 
\ell_{4,1} & \text{if $x_1 > \theta_{4,1}$}\\
\ell_{4,1} + \delta_{4,1} & \text{if $x_1 < \theta_{4,1}$}.
\end{cases}
\end{align*}

  This change of inequality persists into the  functions
  $\Lambda_4(\sigma^+(x_2;\beta_{1,2})x_1)$ under the  assumption {\bf A3.3} and to function $\Lambda_4(\sigma^-(x_3;\beta_{1,3})x_1)$  under the assumption {\bf A3.4}, resulting in 
  \[
      \dot{x}_4 =  -\gamma_4 x_4 +  [\lambda^-(x_1),\lambda^+(x_2)]
  \]
  and 
   \[
      \dot{x}_4 =  -\gamma_4 x_4 +  [\lambda^-(x_1),\lambda^-(x_3)],
    \]
  respectively.

     For interactions in the left side of  Figure~\ref{fig:RN+}(d) it is the $x_1^\mathsf{off}$ form of $x_1$ that  activates production 
     of $x_4$. Since $x_1^\mathsf{on} + x_1^\mathsf{off} = x_1$ under the assumption {\bf A3.3} we get 
     \[\dot{x}_4 = -\gamma_4 x_4 + \Lambda_4(\sigma^-(x_2;\beta_{1,2})x_1) = -\gamma_4 x_4 +  [\lambda^+(x_1),\lambda^-(x_2)] ,\]
     and under the assumption {\bf A3.3} 
\[
\dot{x}_4 = -\gamma_4 x_4 + \Lambda_4(\sigma^+(x_3;\beta_{1,3})x_1 = -\gamma_4 x_4 + [\lambda^+(x_1),\lambda^+(x_3)]
\]

Analogous to Figure~\ref{fig:RN+}(c), interactions indicated by   the right half of Figure~\ref{fig:RN+}(d) will produce functions
     \[  [\lambda^-(x_1),\lambda^-(x_2)] \quad \mbox{ and } \quad   [\lambda^-(x_1),\lambda^+(x_3)]
     \]
respectively.

We close this section with a remark that when $y_1^{\mathsf{on}}$ regulates activity, rather than abundance, of protein $y_2$ which in turn regulates activity of $y_3$ along a chain of interactions until $y_{n-1}$, which eventually regulates transcription of $y_n$, then the equation for $y_n$ will be
\begin{equation}
\label{composition}
\dot{y}_n = -\gamma_n y_n + \Lambda( (\sigma^\pm_{n-1}\circ \ldots \circ \sigma^\pm_1) (y_1^{\mathsf{on}}; \beta_{1,2})y_n),
\end{equation}
where $(\sigma^\pm_{n-1}\circ \ldots \circ \sigma^\pm_1) (y_1^{\mathsf{on}};\beta_{1,2})$ is the $(n-1)$-fold composition of functions $\sigma^+(\; \boldsymbol{\cdot} \; ; \beta_{1,2})$ or $\sigma^-(\; \boldsymbol{\cdot} \; ; \beta_{1,2})$, as functions of the first argument for a fixed value of the parameter $\beta_{1,2}$, depending on the type of control in Figure~\ref{fig:RN+}(c) or (d).

We generalize the function
\[
\Gamma(\zeta_1, \ldots, \zeta_n) := \min \{ \zeta_1, \ldots, \zeta_n \}
\]
and (note identical values of $\ell$ and $\delta$) set
\begin{align*}
[\lambda(z_1),\lambda(z_2),\ldots,\lambda(z_{n}))] 
& :=  \Gamma(\lambda(z_1, \ell, \delta, \theta_1),\lambda(z_2, \ell, \delta, \theta_2),\ldots,\lambda(z_{n}, \ell, \delta, \theta_{n})).
\end{align*}
Then equation \eqref{composition} becomes
\[
\dot{y}_n = -\gamma_n y_n + [\lambda^\pm(y_1),\lambda^\pm(y_2),\ldots,\lambda^\pm(y_{n})],
\]
where the sign of the function $\lambda^\pm(y_j)$ matches the sign of $\sigma^\pm_j(y_j;\beta_{j,j+1})$ for $j=1, \ldots, n-1$, and the sign sign of $\lambda^\pm(y_n)$ matches the sign of the function $\Lambda(y_n)$. Similarly, the threshold parameters $\theta_j$, $j=1, \ldots, n-1$, are the  activation parameters $\beta_{j,j+1}$ and $\theta_n$ is the threshold of the function $\Lambda(y_n)$.

\section{Conclusions}
\label{sec:conclusions}

Networks are a useful abstraction expressing partial knowledge about internal correlation (undirected edges) or causal (directed edges) structure of complex systems. In addition, in gene regulatory networks edges are directed and signed, where sign denotes up- or down-regulation. 
While networks sometimes express static information like correlations, often one is interested in the dynamical behavior of the system described by the network. One of the ways to associate dynamics to a network is to represent each vertex by a continuous variable with a linear decay rate, each edge by a (nonlinear) monotone function, where the sign of the derivative matches the sign of the edge, and study a set of ordinary differential equations with this  structure. One can view such a system as one where variables represent \textit{abundance} with a rate of change responding to associated in-edges. In the context of gene regulation, this type of model corresponds to transcriptional regulation of genes, where an increase in the concentration of activators increases the rate of production from a particular gene, and an increase in the concentration of repressors decreases such a rate. 

Understanding the dynamics of such systems, especially in systems with several to dozens of variables, is notoriously difficult. In particular, the dynamics of network models can vary widely with selection of nonlinearities as well as parameters, which are mostly unknown and lie in high dimensional space. Motivated by gene networks, and building on previous work on Boolean networks and switching systems~\cite{albert:collins:glass,Glass1972,Thomas1991,Ironi2011,edwards00,deJong2002}, we developed DSGRN. DSGRN assigns combinatorial dynamics to a network, of a type depending on a finite decomposition of the parameter space \cite{cummins:gedeon:harker:mischaikow:mok,Cummins2017b,gedeon:harker:kokubu:mischaikow:oka,gameiro:gedeon:kepley:mischaikow,gedeon:2020,gedeon:cummins:harker:mischaikow:plos}. The finiteness of this calculation allows complete enumeration of types of dynamics compatible with the network. However, in spite of the finiteness of representation, the computed invariant, the Morse graph, is valid for a large class of ODE models. In particular, it has been shown that the Morse graph provides information about a Morse decomposition of nearby smooth systems of differential equations in $2$ dimensions~\cite{gedeon:harker:kokubu:mischaikow:oka}, the generalization of this result to higher dimensions is forthcoming. In the case where the Morse nodes indicate the presence of stable equilibria, these equilibria do exist for nearby smooth differential equations~\cite{duncan1,duncan2}.

Again motivated by gene regulation, this paper extends the combinatorial DSGRN approach to a significantly larger class of network interactions. In cellular regulatory networks, the abundance of a particular protein may be constant, but its \emph{activity} may be carefully regulated by say, phosphorylation, methylation or other type of post-transcriptional or post-translational regulation. In addition, the decay of a protein may be actively regulated as well. 
We use modeling work on multi-site phosphorylation and ubiquitination to derive an appropriate combinatorial model for activity and decay regulation.
DSGRN relies on precomputed logic files that encode all linear orders of a set of polynomials associated to a network vertex with a particular number of in- and out-edges. Based on our analysis we show how we modify and then use these precomputed logic files to support combinatorial model for activity and decay regulation.

Finally, we provide a comparison of dynamics between networks that include either activity or decay regulation, and the corresponding networks with the same types of edges but that only regulate abundance. It is clear from these results that the dynamics can be very different and therefore the new capability will allow a more precise delineation of network dynamics in a wide range of applications. 

\section*{Acknowledgments}
B.C. and T.G. were  partially supported by  NSF grant DMS-1839299,  DARPA FA8750-17-C-0054 and NIH 5R01GM126555-01.
The work of M.G., S.K., K.M., and L.Z. was partially supported by the National Science Foundation under awards DMS-1839294 and HDR TRIPODS award CCF-1934924, DARPA contract HR0011-16-2-0033, and National Institutes of Health award R01 GM126555. K.M. is also supported by a grant from the Simons Foundation. The work of M.G. was also partially supported by FAPESP grant 2019/06249-7 and by CNPq grant 309073/2019-7.

B.C. and T.G. acknowledge the Indigenous nations and peoples who are the traditional owners and caretakers of the land on which this work was undertaken at Montana State University.

\bibliographystyle{plain}
\bibliography{references,ptm_ref}

\end{document}